\documentclass[ijoc,nonblindrev]{format}
\usepackage[english]{babel}
\usepackage[utf8]{inputenc}

\OneAndAHalfSpacedXII 

\usepackage{natbib}

  \newcommand{\Prob}{\mathbb{P}}

\usepackage{preamble}

\newcolumntype{L}[1]{>{\raggedright\arraybackslash}p{#1}}
\newcolumntype{C}[1]{>{\centering\arraybackslash}p{#1}}
\newcolumntype{R}[1]{>{\raggedleft\arraybackslash}p{#1}}

\TheoremsNumberedThrough     
\ECRepeatTheorems

\EquationsNumberedThrough    

\MANUSCRIPTNO{JOC-2022-03-OA-067} 

\begin{document}

\RUNAUTHOR{Boonstra, van Eekelen, and van Leeuwaarden}

\RUNTITLE{Robust knapsack ordering for a partially-informed newsvendor}

\TITLE{
Robust knapsack ordering for a partially-informed newsvendor with  budget constraint}

\ARTICLEAUTHORS{%
\AUTHOR{Guus Boonstra}
\AFF{Retail Consulting Department, IG\&H Consultants, \EMAIL{guus.boonstra@igh.com}} 
\AUTHOR{Wouter J.E.C. van Eekelen}
\AFF{Department of Econometrics and Operations Research, Tilburg University, \EMAIL{w.j.e.c.vaneekelen@tilburguniversity.edu}} 
\AUTHOR{Johan S.H. van Leeuwaarden}
\AFF{Department of Econometrics and Operations Research, Tilburg University,  \EMAIL{j.s.h.vanleeuwaarden@tilburguniversity.edu}}
} 

\ABSTRACT{%

This paper studies the multi-item newsvendor problem with a constrained budget and information about demand limited to its range, mean and mean absolute deviation.
We consider a minimax model that determines order quantities by minimizing the expected overage and underage costs for the worst-case demand distributions.
The resulting optimization problem turns out
to be solvable by a method reminiscent of the greedy algorithm that solves the continuous knapsack problem, purchasing items in order of marginal value. 
This method has lower computational complexity compared to directly solving the model and leads to a simple policy that (i) sorts items based on their marginal effect on the total cost and (ii) determines order quantities according to this ranking until the budget is spent.
%
}



\HISTORY{This paper was first submitted on March 8, 2022.}

\maketitle





\section{Introduction}\label{sec:intro}

The newsvendor model is one of the cornerstones of inventory management, introduced by \citet{arrow1951optimal} for finding the order quantity that minimizes expected costs in view of unknown demand and the trade-off between leftovers and lost sales.
The newsvendor model finds many applications in e.g.~perishable food, fashion and high-tech industries, particularly when the total time span of production and lead times exceeds the market lifetime of a product; see~\citet{nahmias1982perishable} and \cite{fisher1996reducing}. 

Manufacturers and retailers need to decide how to employ the available budget or resources when determining the optimal order quantities of different products.
A budget constraint makes the problem multidimensional---as ordering more of one item leaves less budget for other items---and gives rise to a challenging optimization problem. \citet{Hadley1963} solve this problem with Lagrangian optimization. \citet{abdel2004exact} and  \citet{lau1996newsstand} provide alternative solution methods, \citet{erlebacher2000optimal} establishes closed-form solutions for special demand distributions and  \citet{nahmias1984efficient} develop heuristic solutions. All these works are for the full information setting, where the demand distributions for all items are fully specified. 
In this paper we perform a distribution-free analysis of the multi-item newsvendor problem with budget constraint. This analysis does not rely on full specification of the demand distributions, but only requires for each item knowledge of the mean, mean absolute deviation (MAD) and range. Given this partial demand information, we obtain a robust ordering policy by employing distributionally robust optimization (DRO) methods. 

The newsvendor model in this paper seeks to minimize the expected costs as function of the order quantity. 
The cost function depends on the order quantity, but also on the demand, which is a random variable with some distribution. 
Given the demand distribution, the single-item newsvendor model finds the optimal order quantity that minimizes the expected costs. 
In traditional approaches, the demand distribution is fully specified, so that the expected costs can be calculated, and the optimal order quantity can be determined. 
A {\bf \it robust version} of this problem assumes partial information, and only knows that the demand distribution belongs to some ambiguity set 
that contains all distributions that comply with this partial information. We adopt a minimax strategy that can be viewed as a game between the newsvendor and nature: the newsvendor first picks the order quantity after which nature chooses a demand distribution that maximizes the expected costs.
The goal then becomes to solve this minimax problem.

The way we solve this minimax problem in this paper fits in a much richer class of DRO approaches that first calculate worst-case model performance, over the set of distributions satisfying some partial information, and then optimize against these worst-case circumstances. 
Such DRO techniques found applications in many domains including scheduling \citep{kong2013scheduling,mak2014appointment}, portfolio optimization \citep{popescu2007robust,Delage2010}, pricing \citep{elmachtoub2021value,chen2022distribution,kleerleeuwaarden2022}, complex networks \citep{stegehuis2021robust}, and 
  inventory management \citep{Scarf1958, gallego1992minmax, perakis2008regret, ben2013robust}. A classic distributionally robust approach is due to \cite{Scarf1958}, who considered the single-item newsvendor problem with mean-variance demand information. 
 Scarf was able to derive explicit expressions for the worst-case distribution, and solved the minimax problem to obtain the optimal order quantity. Whether a minimax problem 
 is solvable depends on both the function to be  optimized and the choice of ambiguity set. There are many ways to characterize a set of distributions. In DRO, one can define ambiguity by using distance-based metrics, such as total variation or Kullback-Leibler distance. 
 Another popular class of ambiguity uses summary statistics. The ambiguity set studied in this paper contains all distributions with known mean and MAD. The maximization part of the minimax problem can then be viewed as a semi-infinite linear optimization problem
with three constraints, and an infinite number of variables (all distributions in the ambiguity set). 
In fact, such minimax problems are related to generalized moment bound problems, for which general theory says there exists an extremal distribution solving the maximization part with at most a number of support points equal to the number of moment constraints \citep{rogosinski1958moments}. 
See \cite{rahimian2019distributionally} for overviews of many more DRO applications and techniques.


For the multi-item newsvendor model in this paper, we solve the multi-dimensional minimax problem
with 
a random vector that describes the demand for all items. 
Compared with tractable one-dimensional problems such as the single-item newsvendor model, applying DRO techniques to such problems with multiple random variables might present considerable challenges in terms of computational complexity. For example, given information on the mean and covariance of the demands, the distributionally robust multi-item newsvendor is significantly harder to solve than its single-item counterpart \citep{hanasusanto2015distributionally}.
However, for the multi-item newsvendor model in conjunction with mean-MAD ambiguity, solving the minimax problem becomes tractable, and in fact has an elegant algorithmic solution. The key insight will prove to be that the worst-case demand distribution---the solution to the maximization part of the minimax problem---is identical for any order quantity. As a result, the minimax 
problem  reduces to a known-distribution optimization problem. This known distribution is in fact, for each item, a unique three-point distribution. In turn, the minimization problem with this known (discrete) distribution can be solved using a reduction to a knapsack problem.   

The main contributions of this paper are as follows:
\begin{itemize}
    \item[(i)] {Solution of minimax problem.}
We solve the minimax problem for mean-MAD ambiguity and a budget constraint. We first show that the worst-case scenarios arise when item demands follow specific three-point distributions that comply with the partial demand information. We minimize the associated worst-case costs to obtain a robust ordering policy as the solution to a knapsack problem. As opposed to existing methods for the newsvendor model under full demand information, the knapsack problem leads to an effective closed-form ordering policy, also for scenarios with many items. As such, the present paper further develops DRO theory that uses MAD information to formulate tractable minimax problems. 
\item[(ii)] {Budget consistency.} The robust ordering policy only depends on the minimal, mean and maximal demand for each item. Hence, the worst-case distributions are independent of all other model parameters, which makes the robust ordering policy `budget consistent'. When the budget is increased, 
the orders for the original budget remain unaltered, while only the additional budget is further divided over the items. Such budget consistency 
 is useful because the optimization model needs to be solved only once. That is, for the initial budget value the decision maker can generate an ordered list of items as the solution to the knapsack problem, using only standard spreadsheet software, and this solution is valid for all budget levels. In contrast, most other exact and robust methods for the multi-item newsvendor model do not have this feature, which means that the decision maker has to recompute the optimal policy for each budget level.

 \item[(iii)] {Performance of ordering policy.} Through a range of numerical examples we demonstrate the performance of the knapsack ordering. We draw comparisons with full information settings and other robust approaches that require partial demand information by assessing the so-called expected value of additional information (EVAI). Overall, the performance of the robust policy
only deviates a few percent from the optimal performance with full information availability. 
We also quantify the value of MAD information by comparing the performance with
the situations when only the mean and range of demand is known, and show that MAD indeed provides crucial information for providing good performance. 
In addition, we construct an ordering policy that attains the optimal value of a matching minimin problem which, in conjunction with the optimal value of the minimax problem, yields tight performance guarantees.

\end{itemize}

We next discuss some related literature on the newsvendor model.
\cite{gallego1993distribution} consider the multi-item newsvendor model with budget constraint when the mean and variance of demand is known. 
\cite{gallego1993distribution} extend the ideas in \citet{Scarf1958}  to
obtain an optimization problem that can be solved with Lagrange multiplier techniques, similar to the full information setting with a known distribution. In contrast, our minimax analysis with mean-MAD-range information yields a knapsack ordering policy that  generates a sorted list  and prescribes to sort items successively according to that list, with order sizes equal to the minimal, mean or maximum demand.
Other related works that consider the multi-item newsvendor model under partial information include \citet{vairaktarakis2000robust}, who assumes only the support of demand is known, and \cite{ardestani2016robust} who assume knowledge of partial moments and rephrase the robust optimization problem as a tractable linear program. 
 \citet{natarajan2018asymmetry} assume knowledge of mean, variance and semivariance, for which the newsvendor model is solvable in the single-item setting using a semi-infinite linear program, 
 but largely intractable in the multi-item setting. \citet{natarajan2018asymmetry} therefore consider a relaxation that gives a semidefinite program (SDP) to find a lower bound (which is not tight). 
 \citet{hanasusanto2015distributionally} 
consider mean and covariance knowledge. They prove that the distributionally robust problem is NP-hard but admits a semidefinite programming formulation with an exponential number of inequalities (that grows in the number of items).  \cite{xu2018distributionally} and \cite{natarajan2017reduced} present more tractable bounds for mean-covariance information. 
In the present paper we assume only marginal information is available, since covariance information and other dependency structures are difficult to estimate, and  fixing covariance information often leads to difficult optimization problems with non-intuitive solutions (policies). 
The knapsack ordering policy that we obtain in this paper
deals with the worst-case demand distributions among all demand distributions with a given mean, MAD and range, not conditioning on a specific dependency structure. This approach makes the knapsack ordering policy robust, but also suitable for scarce-data settings, as the mean, MAD and range are relatively easy to estimate. 

Section~\ref{sec:2} introduces the single-item model and the multi-item model with budget, under the traditional assumption of full information about the demand distributions. In Section~\ref{sec:3} we present our main results for the distributionally robust setting with partial information. 
Section~\ref{sec:numm} presents a detailed numerical study that demonstrates the robust policies. 
We present conclusions and several directions for future work in Section~\ref{sec:outlook}.  Supplementary material appears in the Electronic Companion (EC), including several proofs, additional numerical experiments, and model extensions.

\section{Classical newsvendor analysis}\label{sec:2}

We introduce the newsvendor model and several well-known results in Section~\ref{sec:model} for the single-item setting, and in 
Section~\ref{sec:multi} for the multi-item setting  with budget constraint. 

\subsection{Classical single-item setting}\label{sec:model}
Consider an item with purchase price $c$ and selling pricing $p$. The decision maker places an order of size $q$. The demand for items is assumed to be the random variable $D$ with distribution function $F_D(\cdot)$. Unsold items will be salvaged at the end of the period for salvage value $s$ per item. 
The mark-up $m>0$ represents the profit per sold item and satisfies $p=c(1+m)$ and the discount factor $d>0$ captures the loss through $s=(1-d)c$. 

The expected costs consist of two terms: opportunity costs of lost sales and overage costs in case of overstocking. 
This gives the cost function
\begin{equation}\label{eq:ObjFuncUni}
G(q,D) =     \begin{cases}
\begin{array}{ll}
(p-c)(D-q)  \quad \text{if } q \leq D, \\
(c-s)(q-D)  \quad \text{if } q > D.
    \end{array}
\end{cases}
\end{equation} 
The case $q \leq D$ amounts to lost sales and $q > D$ results in overstocking. The objective is to order the quantity $q$ of items that minimizes the expected costs. Let $\E$ denote expectation, and define $\mu=\E[D]$ and $x^+= \max(x,0)$. Write the expected costs as
\begin{equation}\label{eq:ObjFuncUni2}
     C(q) := \E[G(q,D)] = (c-s)q + (p-s)\mathbb{E}(D-q)^+ - (c-s)\mu  = c\left(d(q-\mu) + (m+d)\E(D-q)^+ \right). 
\end{equation} 
To keep notation simple (and without loss of generality) set  $c=1$.
Then, the optimal order quantity
\begin{equation}\label{model:newsvendorSingleItem}
q^*=\argmin_{q \geq 0} C(q) \equiv  \argmin_{q\geq0} dq + (m+d)\mathbb{E}(D-q)^+, 
\end{equation}
is given by 
\begin{equation}
    q^* = \inf\Big\{q \,:\, F(q)\geq \frac{m}{m+d}\Big\}. 
    \label{eq:qU}
\end{equation}
A proof of \eqref{eq:qU} is provided in most standard textbooks on inventory management; see e.g. \citet{Hadley1963,silver1998inventory,nahmias2009production}. 

\subsection{Multi-item setting}\label{sec:multi}
Consider $n$ different items and order $q_i$ units for item $i$ for a given period where $i=1,\dots,n$. For item $i$, the unit purchasing and selling price are $c_i$ and $p_i$ respectively. Possible leftovers will be salvaged at the end of the period for unit salvage value $s_i$. We define the model in terms of the mark-up $m_i>0$ and discount factor $d_i>0$. The mark-up represents the profit per sold unit and the discount factor the loss, i.e.\ $p_i=c_i(1+m_i)$ and $s_i=(1-d_i)c_i$. 
The random demand for item $i$ in one period is represented by the nonnegative random variable $D_i$, distributed according to $F_i(\cdot)$. 

As in the single-item setting, we minimize the expected costs. 
Define the multi-item cost function as
\begin{equation}\label{eq:multicostfunction}
G(\vecc{q},\vec{D}) := \sum_{i=1}^{n} c_i\left(d_i(q_i-D_i) + (m_i+d_i)(D_i-q_i)^+\right). 
\end{equation}
We also introduce the budget constraint 
$\sum_{i=1}^n c_iq_i \leq B$ with $B$ the available budget. The multi-item newsvendor model, with decision vector $\vecc{q}=(q_1,\ldots,q_n)$, is then given by 
\begin{equation}
\begin{aligned}
\min_{\vecc{q}}& \quad C(\vecc{q}):= \E[G(\vecc{q},\vec{D})] = \sum_{i=1}^n c_i\left(d_i (q_i-\mu_i) + (m_i+d_i)\mathbb{E}(D_i-q_i)^+\right)
\\
\textrm{s.t.}& \quad   \sum_{i=1}^n c_i q_i \leq B, \\
& \quad q_i \geq 0, \quad i=1,\dots,n.
\end{aligned}
\label{model:CapNewsModel}
\end{equation}
Its solution, referred to as the optimal ordering policy, will be denoted by $\vecc{q}^*$. 
In the single-item setting the purchase costs had no influence on the objective function, but in the multi-item setting the optimal order quantity is affected by $c_i$. It is well known that model \eqref{model:newsvendorSingleItem} is a convex optimization problem. In \eqref{model:CapNewsModel} we take the summation over $n$ convex functions, which preserves convexity. Moreover, the constraints 
form a convex set, so  that \eqref{model:CapNewsModel} is a convex optimization problem \citep{boyd2004convex}. 

\section{Proposed robust approach}\label{sec:3}
Section~\ref{sec:meanMADorder} presents the robust ordering policy for the single-item setting. 
 This result serves as building block for the robust analysis of the multi-item setting in Section~\ref{subsec:3}, which describes the optimal policy as the solution of a linear program (LP).
 In Section~\ref{sec:allocalgo} we show that this LP can be viewed as a knapsack problem.
 All these results are based on a tight upper bound for the cost function. In Section~\ref{sec:variancebeta} we derive a matching tight lower bound for the cost function.

\subsection{Distribution-free ordering policy for single item}\label{sec:meanMADorder}
Let $\P$ denote a probability distribution, and write $\E_{\Prob}$ for $\E$ to emphasize that the expectation is taken with respect to the distribution $\Prob$ of $D$.
The MAD for random demand $D$ is defined as $\delta :=  \mathbb{E}_\P|D-\mu|$, where $\mu$ is the expected value of $D$. Similar to the variance, the MAD is a measure of dispersion or variability. We mention several properties of MAD in \ref{ec:propertiesMAD}. 
For the random variable $D$ with mean $\mu$, MAD $\delta$, and (bounded) support $[a,b]$, where $0\leq a\leq b<\infty$, the mean-MAD ambiguity set is defined as
$$
\mathcal{P}_{(\mu,\delta)} := \left\{ \mathbb{P} \, | \, \mathbb{E}_{\mathbb{P}}[D] = \mu, \, \mathbb{E}_{\mathbb{P}}|D-\mu|= \delta, \, \text{supp}(D) \subseteq [a,b] \right\}.
$$
We thus assume that the ‘true’ distribution $\tilde{\mathbb{P}}$ of the random demand $D$ is contained in this ambiguity set, that is, $\tilde{\mathbb{P}} \in \mathcal{P}_{(\mu,\delta)}$.

To obtain the robust order quantity, we solve
\begin{equation*}
    \min_q \max_{\mathbb{P} \in \mathcal{P}_{(\mu,\delta)}} dq + (m+d)\mathbb{E}_{\mathbb{P}}(D-q)^+,
\end{equation*}
{for which we first consider 
$\max_{\mathbb{P} \in \mathcal{P}_{(\mu,\delta)}} \mathbb{E}_{\mathbb{P}}(D-q)^+$. }
%
To characterize this tight bound,  we apply a general upper bound for convex functions of a random variable by \citet{ben1972more}. To make this paper self-contained, we provide a proof of the following result in \ref{ec:proofs}. 
\begin{lemma}\label{lemma:meanmad1}
The extremal distribution that solves 
$\max\limits_{\mathbb{P} \in \mathcal{P}_{(\mu,\delta)}} \mathbb{E}_{\mathbb{P}}(D-q)^+$
is a three-point distribution on the values $a,\,\mu$ and $b$ that does not depend on $q$.
\end{lemma}

From the proof of Lemma~\ref{lemma:meanmad1}, it follows that the worst-case probability distribution of $D$, the extremal distribution that solves $\max_{\mathbb{P} \in \mathcal{P}_{(\mu,\delta)}}\mathbb{E}_{\mathbb{P}}(D-q)^+ $, is a three-point distribution defined as 
\begin{equation}\label{eq:worstcase3pointdist}
    \mathbb{P}(D=x) = \left\{
    \begin{array}{ll}
        \dfrac{\delta}{2(\mu-a)}, & \text{for }  x = a, \\
        1- \dfrac{\delta}{2(\mu-a)}- \dfrac{\delta}{2(b-\mu)}, & \text{for }   x=\mu, \\
        \dfrac{\delta}{2(b-\mu)}, & \text{for }  x = b. \\
    \end{array}
\right.
\end{equation}

Applying this worst-case distribution, the robust order quantity follows from solving $q^U=\argmin_q C^U(q)$ with
\begin{equation}\label{model:newsvendorSingleMeanMad}
\begin{aligned}
    C^U(q) &  :=  d(q-\mu)  + \dfrac{\delta(m+d)}{2(\mu-a)}(a-q)^+ +  (m+d)\left(1- \dfrac{\delta}{2(\mu-a)}- \dfrac{\delta}{2(b-\mu)}  \right)(\mu-q)^+ \\ & + \dfrac{\delta(m+d)}{2(b-\mu)}(b-q)^+.  
\end{aligned}
\end{equation}
To illustrate the mean-MAD bound and robust order quantity $q^U$, consider an example in which $D$ is distributed according to a beta distribution with both shape parameters set to 1. For a general beta distribution, $a=0$ and $b=1$. In Figure~\ref{fig:exMeanMadSingle1}, we have $m = 1$ and $d=0.8$. This leads to $q^{U} = \mu $. In Figure~\ref{fig:exMeanMadSingle2}, the mark-up increases to $m=3$. In this case the mean-MAD order quantity increases to $q^{U} = b$. 
\begin{figure}[h]
\centering
\begin{subfigure}{.5\textwidth}
  \centering
  \includegraphics[width=\linewidth]{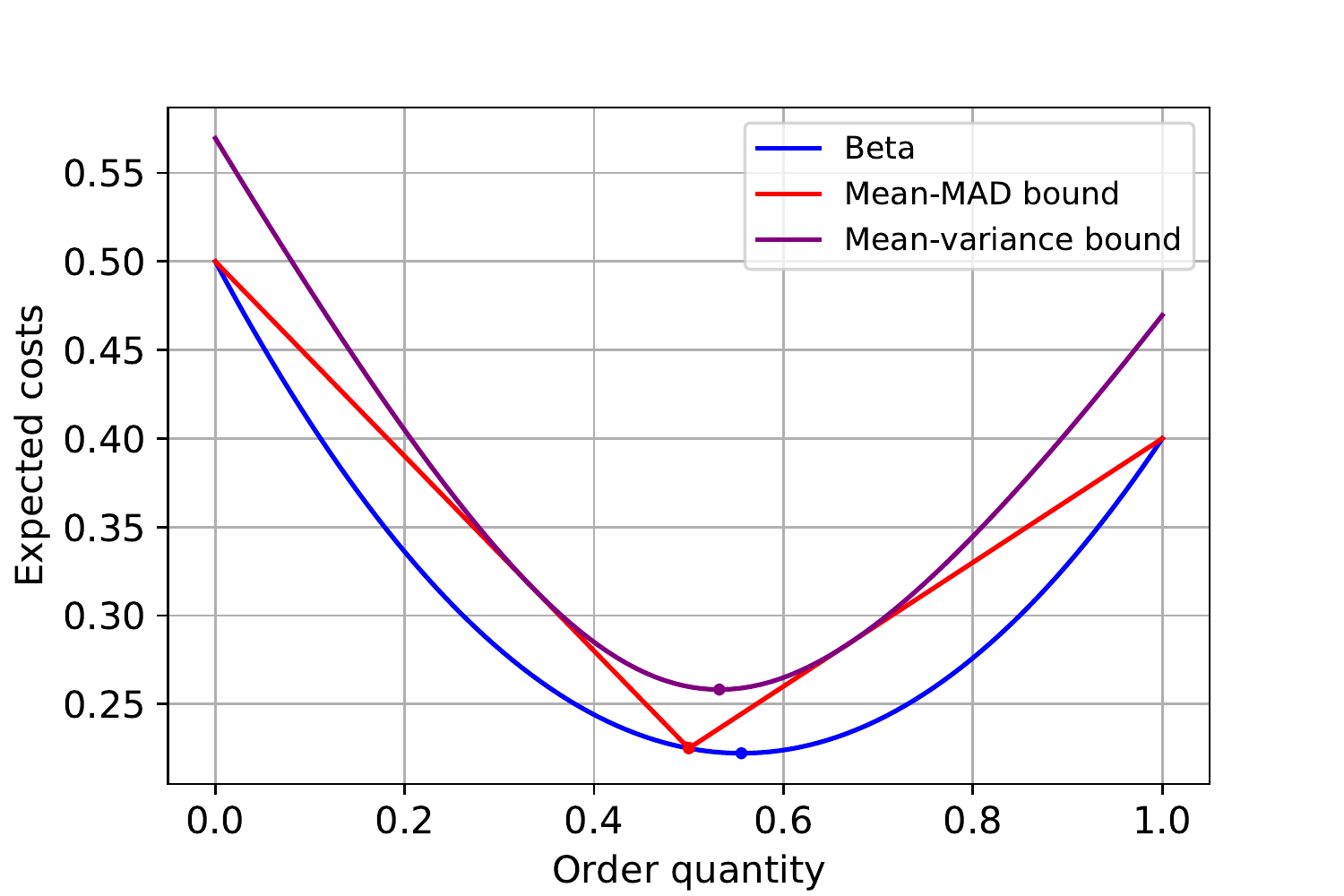}
    \caption{$m=1$}
  \label{fig:exMeanMadSingle1}
\end{subfigure}%
\begin{subfigure}{.5\textwidth}
  \centering
  \includegraphics[width=\linewidth]{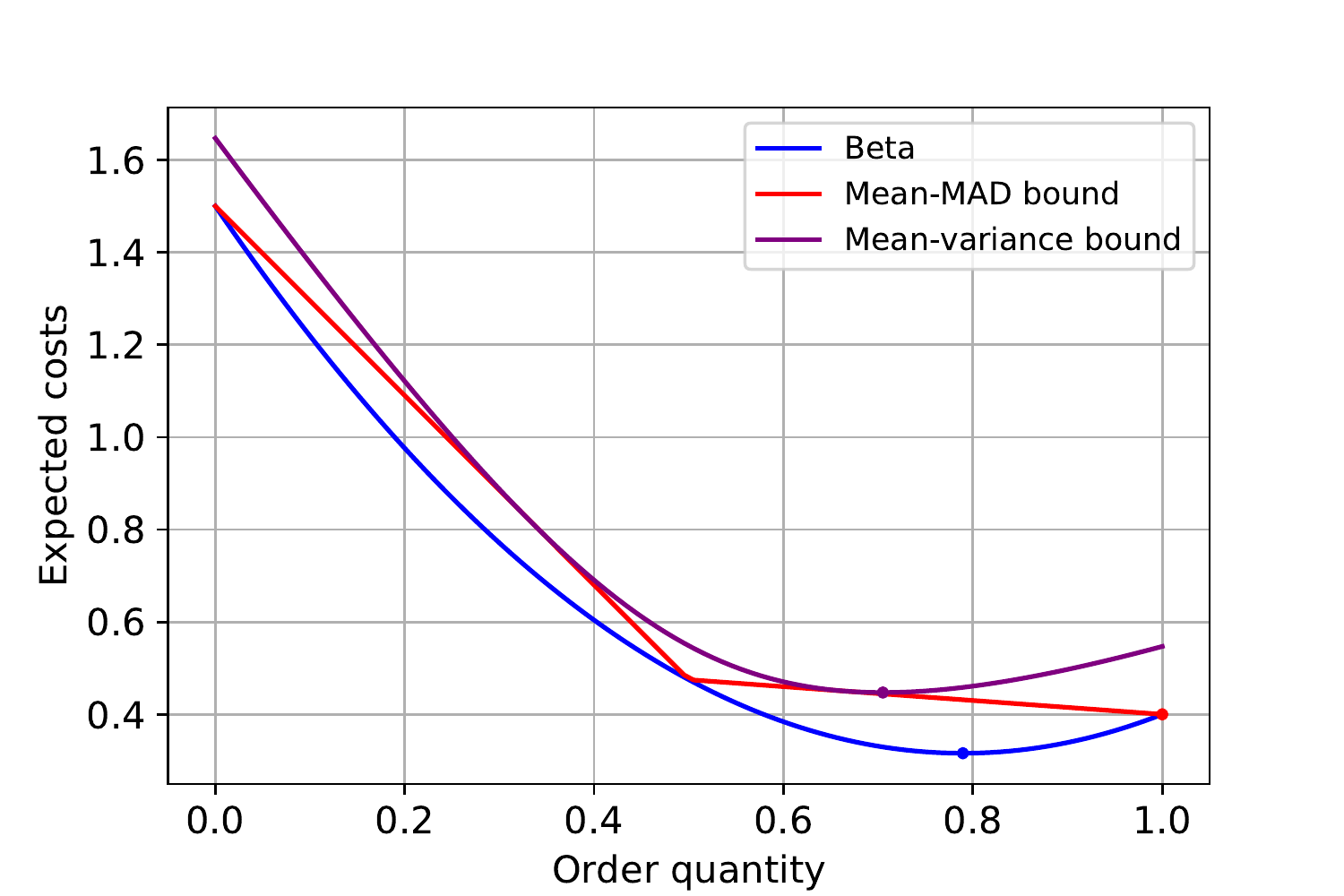}
  \caption{$m=3$}
  \label{fig:exMeanMadSingle2}
\end{subfigure}
\caption{Mean-MAD and mean-variance bounds and corresponding ordering policies. The upper curve corresponds to the mean-variance upper bound that follows from $\cP_{(1/2,1/12)}$. The middle curve depicts the mean-MAD upper bound. The ‘true’ cost function assumes that $D$ follows a beta distribution with both shape parameters equal to 1 (the lower curve).}
\end{figure}
When computing this upper bound, observe that the mean-MAD bound touches the ‘true’ cost function in the points $a,\mu$ and $b$. This property actually holds in general. Clearly, for $q=a$ or $b$, it holds that $C^U(q)=C(q)$. When $q=\mu$, the cost function equals
$$
C(\mu)= d(\mu-\mu)+(m+d)\E(D-\mu)^+ =   \dfrac{\delta(m+d)}{2} = C^U(\mu),
$$
since $\E(D-\mu)^+=\E|D-\mu|/2$.

By analyzing \eqref{model:newsvendorSingleMeanMad} one can obtain an explicit ordering rule for $q^U$. The objective function of \eqref{model:newsvendorSingleMeanMad} is composed of piecewise linear functions. 
By exploiting this structure, we can construct an explicit ordering policy. 
For scalars $\alpha_1,\dots,\alpha_m,\nu_1,\dots,\nu_m\in\mathbb{R}$,
$f(x) = \max_{i=1,\dots,m}\{\alpha_i x + \nu_i\}$ denotes a convex, piecewise linear function.
The function $C^U(q)$ in~\eqref{model:newsvendorSingleMeanMad} admits a representation of the form
$$
\begin{aligned}
C^U(q) =  d(q-\mu) + (m+d)\E(D-q) 
    =  m(\mu-q) \eqqcolon f_0(q),
\end{aligned}   
$$
for $q\in[0,a)$ and 
\begin{equation*}
\begin{aligned}
C^U(q) &=  d(q-\mu) + (m+d)\left(1- \dfrac{\delta}{2(\mu-a)}- \dfrac{\delta}{2(b-\mu)}  \right)(\mu-q) + \dfrac{\delta(m+d)}{2(b-\mu)}(b-q) \\
    &= q(\dfrac{\delta(m+d)}{2(\mu-a)}-m) + \nu_1 \eqqcolon f_1(q),
\end{aligned}    
\end{equation*}
for  $q \in [a, \mu)$, 
where $\nu_1$ is some constant value. 
For $q \in [a,\mu)$, the mean-MAD objective function is defined by the linear function $f_1(q)$. For the interval $q \in [\mu,b]$, we obtain
\begin{equation*}
     C^U(q) =  d(q-\mu) + \dfrac{\delta(m+d)}{2(b-\mu)}(b-q) = q\left(d- \dfrac{\delta(m+d)}{2(b-\mu)}  \right) + \nu_2 \eqqcolon f_2(q)
\end{equation*}
for some constant $\nu_2$. The cost function is thus the pointwise maximum of the three linear functions $f_0(q)$, $f_1(q)$ and $f_2(q)$:
\begin{equation*}\label{replin3}
    C^U(q) = \max\left\{f_0(q),\,f_1(q),\,f_2(q)\right\}.
\end{equation*}
Since $C^U(q)=\max_{j=0,1,2} \{\alpha_{j} q + \nu_{j}\} $ is a convex function, it holds that $\alpha_{0} \leq \alpha_{1} \leq \alpha_{2}$. 
 Since we assume that $m > 0$, we know that $\alpha_0 < 0$.
Therefore, from the derivatives $\alpha_{1}$, $\alpha_{2}$ of $C^{U}(q)$, we can derive an explicit order quantity by examining for which linear piece the slope turns positive. This allows us to state Theorem~\ref{thm:singleMadStrategy}.

\begin{theorem}[Mean-MAD order quantity]\label{thm:singleMadStrategy}
The robust order quantity $q^U\in\argmin_{q}C^U(q)$ is given by
\begin{enumerate}[label=\normalfont(\alph*)]
     \item If $m < \dfrac{\delta d}{2(\mu-a)-\delta}    $, then  $q^U= a$. 
    \item If $\dfrac{\delta d}{2(\mu-a)-\delta} < m <  \dfrac{d(2(b-\mu) - \delta)}{\delta}$, then  $q^U = \mu$. 
\item If $ \dfrac{d(2(b-\mu) - \delta)}{\delta} < m $, then $q^U = b$.
\item If $m=\dfrac{\delta d}{2(\mu-a)-\delta}$ and $m=\dfrac{d(2(b-\mu) - \delta)}{\delta}$, then $q^U \in [a,\mu]$ and $q^U \in[\mu,b]$, respectively.
\end{enumerate}
\end{theorem}

According to Theorem~\ref{thm:singleMadStrategy}, the robust order quantity $q^U$ for mean-MAD-range information consists of three predictable values (minimal, mean, maximum demand) that do not depend on the mark-up $m$ and discount factor $d$, whereas the conditions that dictate how much to order do depend on them (in addition to the demand mean, MAD and range).


\subsection{Multiple items and budget constraint}\label{subsec:3}

A distribution-free analysis of the multi-item model requires a multivariate ambiguity set. As in the single-item case, the partial information is the mean $\mu_i$, MAD $\delta_i$ and support  $\operatorname{supp}(D_i) = [a_i,b_i]$ for each random variable $D_i$, $i=1,\ldots,n$. The mean-MAD ambiguity set is defined as
\begin{equation}\label{eq:meanmadambiguity}
\mathcal{P}_{(\mu, \delta)} := \left\{\mathbb{P} \, | \, \mathbb{E}_{\mathbb{P}}\left(D_{i}\right)=\mu_{i}, \, \mathbb{E}_{\mathbb{P}}\left|D_{i}-\mu_{i}\right|=\delta_{i}, \,  \operatorname{supp}\left(D_{i}\right)\subseteq\left[a_{i}, b_{i}\right], \, \forall i\right\}.
\end{equation}
We henceforth assume that the distribution of the vector of random variables $\vec{D} = (D_1,\dots,D_n)$ belongs to this ambiguity set, i.e., $\mathbb{P} \in \mathcal{P}_{(\mu, \delta)}$. Since the objective function in \eqref{model:CapNewsModel} is separable, one can apply the single-item bound to each term $\mathbb{E}\left( D_i - q_i\right)^+$ in the summation individually. The following result, for the multi-item problem, is then a direct consequence of Lemma~\ref{lemma:meanmad1}. 

\begin{lemma}\label{lemma:meanmad}
The extremal distribution that solves 
$\max\limits_{\mathbb{P} \in \mathcal{P}_{(\mu,\delta)}} \expectp{G(\vecc{q},\vec{D})}$
consists for each $D_i$ of a three-point distribution with values $\xi_1^{(i)} = a_i$, $\xi_2^{(i)}  = \mu_i$, $\xi_3^{(i)}  = b_i $ and probabilities
\begin{equation}
\label{eq:BenTal_probabilities_multivariable}
p_{1}^{(i)}  = \frac{\delta_i}{2(\mu_i - a_i)}, \quad p_{2}^{(i)}  = 1- \frac{\delta_i}{2(\mu_i - a_i)}  - \frac{\delta_i}{2(b_i - \mu_i)} , \quad p_{3}^{(i)}  = \frac{\delta_i}{2(b_i - \mu_i)}.
\end{equation} 
\end{lemma}

For the multi-item newsvendor model based on mean-MAD ambiguity, we use Lemma~\ref{lemma:meanmad} to solve the maximization part of
\begin{equation}\label{model:DROMulti}
\min_{\vecc{q}:\sum_i c_iq_i\leq B,q_i\geq0}\, \max_{\mathbb{P} \in \mathcal{P}_{(\mu,\delta)}} \mathbb{E}_{\P}\Big[ \sum_{i=1}^n c_id_i(q_i-\mu_i) + c_i(m_i+d_i)\left( D_i - q_i\right)^+\Big],
\end{equation}
and obtain
\begin{equation}\label{model:ConstrainedMultiMeanMAD}
\begin{aligned}
\min_{\vecc{q}} \quad & \sum_{i=1}^n c_i\left(d_i(q_i-\mu_i) + (m_i+d_i)\left(p^{(i)}_{1}\left(a_i-q_i \right)^+ + p^{(i)}_{2}\left(\mu_i-q_i \right)^+ + p^{(i)}_{3}\left(b_i-q_i \right)^+\right)  \right)
\\
\textrm{s.t.} \quad &  \sum_{i=1}^nc_iq_i \leq B,
\\ & q_i \geq 0, \quad i=1,\dots,n.
\end{aligned}
\end{equation}
The objective function of \eqref{model:ConstrainedMultiMeanMAD}  has a piecewise linear structure. Moreover, because of this result and since the constraints are linear, \eqref{model:ConstrainedMultiMeanMAD} can be cast as a linear program (LP). 
In particular, as explained below, the robust ordering policy $\vecc{q}^U$ can be found by solving
\begin{equation}\label{model:ConstrainedLPMultiMeanMAD}
\begin{aligned}
\min_{\vecc{q}} \quad & \sum_{i=1}^n \max_{j=0,1,2} \{\alpha_{i,j} q_i + \nu_{i,j}\} 
\\
\textnormal{\textrm{s.t.}} \quad &  \sum_{i=1}^nc_iq_i \leq B,
\\ & q_i \geq 0, \quad i=1,\dots,n , 
\end{aligned}
\end{equation}
where
\begin{gather*}
    \alpha_{i,0} = -c_im_i , \qquad \nu_{i,0} = c_i m_i \mu_i,   \\
    \alpha_{i,1} = c_i\left(\dfrac{\delta_i(m_i+d_i)}{2(\mu_i - a_i)} - m_i\right), \quad   \nu_{i,1} = c_i(m_i+d_i)\left(\mu_i - \dfrac{\delta_i a_i}{2(\mu_i-a_i)}\right)-c_i d_i \mu_i,    \\
    \alpha_{i,2} = c_i\left(d_i - \dfrac{\delta_i(m_i+d_i)}{2(b_i - \mu_i)}\right) , \quad \nu_{i,2} = \dfrac{c_i\delta_i(m_i+d_i)b_i}{2(b_i-\mu_i)}-c_i d_i \mu_i, \quad \text{\rm for }  i=1,\dots,n. 
\end{gather*}
Let $f_{i,j}(x) = \alpha_{i,j}x + \nu_{i,j}$ for $i=1,\dots,n$ and $j=0,1,2$. From the single-item case, we know that the objective, for each item $i$, can be written as $ \max_{j=0,1,2} \{f_{i,j}(q_i)\}$ with $\alpha_{i,0} \leq \alpha_{i,1} \leq \alpha_{i,2}$,
and thus the objective functions of \eqref{model:ConstrainedMultiMeanMAD} and \eqref{model:ConstrainedLPMultiMeanMAD} are equal, which makes the two models equivalent. Since we know from linear programming theory that convex, piecewise linear objective functions can be written as linear constraints, problem \eqref{model:ConstrainedLPMultiMeanMAD} admits an LP representation \citep{boyd2004convex}.

\subsection{Knapsack algorithm}\label{sec:allocalgo}
It turns out that problem \eqref{model:ConstrainedLPMultiMeanMAD} is intimately related to the continuous knapsack problem, thus making available efficient sorting-based algorithms to solve \eqref{model:ConstrainedLPMultiMeanMAD}. We next describe an efficient algorithm that determines the robust ordering policy.

Define the linear funtion $f_{i,j}$ for each item $i$, and let $\alpha_{i,j}$ represent its derivative with respect to $q_i$, for items $i=1,\dots,n$ and linear pieces $j=0,1,2$. That is, 
\begin{equation*}
   \dfrac{\d f_{i,j}(q_i) }{\d q_i} = \alpha_{i,j}.
\end{equation*}

For each item $i$, $f_{i,0}$, $f_{i,1}$ and $f_{i,2}$ represent the marginal effect on the value of \eqref{model:ConstrainedLPMultiMeanMAD} when we increase $q_i$ to $a_i, \mu_i$ and $b_i$ respectively. The parameter $\alpha_{i,j}$ represents the slope of these linear functions and an order quantity is increased only when $\alpha_{i,j} < 0$, because otherwise it will not reduce the expected costs. We consecutively allocate budget to the item that causes the largest relative decrease in expected costs; that is, item $k$ with the smallest negative derivative $\alpha_{k,i}$ relative to its cost $c_k$. Define the set of all items as $N= \{1,\dots,n\}$.
Since only order quantities that decrease the expected costs are considered, define the  ordered set:
\begin{equation}\label{orderStrategy}
        \mathcal{G} := \{(i,j) \mid \alpha_{i,j} < 0, i\in N,  j \in \{0,1,2\} \},
\end{equation}
where the ordering is determined according to the value of $\alpha_{i,j}/c_i$. For $m = |\mathcal{G}|$, this ordering is represented by the sequence $(i_1,j_1),\dots,(i_{m},j_{m})$ for which it holds that
$\alpha_{i_1,j_1}/c_{i_1} \leq \cdots \leq \alpha_{i_m,j_m}/c_{i_m}$. 
Here $\mathcal{G}$ contains tuples $(i,j)$ for which $i$ represents an item in the newsvendor model and $j$ a linear piece of the piecewise function. As these functions are convex, the linear pieces appear for each item $i$ in increasing order in the set  $\mathcal{G}$. We can now state the knapsack algorithm for the distribution-free multi-item newsvendor model.  \\

\begin{algoritme}[Knapsack algorithm]\label{Algorithm1}
\rm For a budget level $B\geq0$, the ordering policy $\vecc{q}^U$ is found by the following procedure:

\begin{enumerate}[label=\normalfont(\roman*)]
\item Initialize by setting $\vecc{q}= (0,\dots,0)$, and construct $\mathcal{G}$. Continue to {\normalfont(ii)}.

\item 
Select the first element $(i,j)\in \mathcal{G}$. If the set $\mathcal{G}$ is empty, the optimal solution is $\vecc{q}^U=\vecc{q}$. Otherwise, continue to {\normalfont(iii)}.
\item 
If $j=0$, set $q_i = a_i$.
If $j=1$, set $q_i = \mu_i$.
If $j=2$, set $q_i = b_i$. 
Continue to {\normalfont(iv)}.
\item 
Determine whether the budget constraint $\sum_{i=1}^n c_iq_i \leq B$ is violated. If so, set $q_i$ such that $c_iq_i = B- \sum_{k\in N | k \neq i} c_kq_k $,
and the optimal solution is $\vecc{q}^U=\vecc{q}$. Otherwise, remove element $(i,j)$ from $\mathcal{G}$ and return to step {\normalfont(ii)}.
\end{enumerate}
\end{algoritme}

This algorithm yields an optimal solution to \eqref{model:ConstrainedLPMultiMeanMAD}, as asserted in the following theorem.

\begin{theorem}[Knapsack ordering policy]\label{thm:AllocationAlgorithm}
The robust ordering policy $\vecc{q}^U$ that solves the multi-item newsvendor model {\normalfont\eqref{model:ConstrainedLPMultiMeanMAD}} is determined by {\normalfont Algorithm~\ref{Algorithm1}}.
\end{theorem}

\begin{proof}{Proof}
To prove that this algorithm produces an optimal solution, we construct a continuous knapsack problem that solves  \eqref{model:ConstrainedLPMultiMeanMAD}.
In the following, $(i_k,j_k)$ corresponds to the $k$th entry of the ordered sequence of items in $\mathcal{G}$. Define the following auxiliary model:
\begin{equation}\label{auxiliaryModelProof}
\begin{aligned}
\min_{\vecc{x}} \quad & \sum_{k=1}^m p_k x_k
\\
\textrm{s.t.} \quad &  \sum_{k=1}^m c_k  x_k  \leq B,
\\ & 0 \leq x_k \leq u_k \quad \forall k=1,\dots,m,
\end{aligned}
\end{equation}
where 
\begin{equation*}
    u_k = \left\{
    \begin{array}{ll}
        a_{i_k}, & \mbox{for } j_k = 0 \\
        \mu_{i_k} - a_{i_k}, & \mbox{for } j_k=1 \\
        b_{i_k} - \mu_{i_k}, & \mbox{for } j_k = 2 \\
    \end{array}
\right.
\end{equation*}
and $p_k = \alpha_{i_k,j_k}$ and $c_k = c_{i_k}$.
From the order of the sequence, it follows that $p_1/c_1 \leq \ldots \leq p_m/c_m$.
Assume that $(x^*_1,\dots,x^*_m)$ is an optimal solution to optimization problem \eqref{auxiliaryModelProof}. For $i \in N$, let $q^U_i = \sum_{k=1,\dots,m| i=i_k }x^*_k$. Since $\alpha_{i,0}\leq\alpha_{i,1}\leq\alpha_{i,2}$, the pieces $j_k$ appear in $\mathcal{G}$ in increasing order for each item $i$. Thus, in an optimal solution, $u_{i_k,j_k}$ will only be attained if its predecessor $u_{i_k,j_l}$ is also attained. By construction, $\vecc{q}^U$ is feasible for \eqref{model:ConstrainedLPMultiMeanMAD}. Moreover, the objective values of problems \eqref{model:ConstrainedLPMultiMeanMAD} and \eqref{auxiliaryModelProof} only differ by a constant term, so both problems have the same optimal solution. 
For the continuous knapsack problem, a greedy allocation produces an optimal solution (see \ref{ec:Knapsack}). Hence, $\vecc{q}^U=(q^U_1,\dots,q^U_n)$ is optimal for \eqref{model:ConstrainedLPMultiMeanMAD}. 
\end{proof}

Theorem 2 shows that there exists a ranking for the selection of items. Take an initial budget $B =0$. If we increase the budget $B$ by some small value, we first increase item $i$ to $a_i$ for the item that has the highest mark-up $m_i$. This makes sense intuitively because the product with the highest mark-up is most profitable and, since $q_i < a_i$, we have no risk of overstocking. We successively select the items with the greatest marginal benefit $\alpha_{i,j}/c_i$, and increase the order quantity consecutively to either $a_i$, $\mu_i$ or $b_i$. 
This procedure continues until we have spent the entire budget, or reached the uncapacitated optimum. Items that are ordered in the beginning of this procedure have the largest impact on the decrease in costs for the multi-item newsvendor model. 

As the main complexity of the knapsack algorithm in Theorem~\ref{thm:AllocationAlgorithm} stems from sorting the set $\mathcal{G}$, the greedy approach is of computational complexity $O(n \log n)$. 
Moreover, the solution can be found in $O(n)$ time by first identifying the critical element $(i_s,j_s)$ that will violate the budget constraint, as proposed by \citet{balas1980algorithm} for the continuous knapsack problem. One then compares each $\alpha_{i,j}/c_i$ with the ratio of the critical element to determine the optimal allocation of budget to the items. 
The optimal solution can also be found through the LP \eqref{model:ConstrainedLPMultiMeanMAD}, which we solve with the simplex method. 
We remark that 
a single iteration of 
the simplex method 
takes $O(n^2)$ arithmetic operations  \citep{illes2002pivot}, which exceeds the time requirement of the knapsack algorithm.

\subsection{A matching lower bound}\label{sec:variancebeta}
The robust analysis so far was based on finding a tight upper bound on the cost function when we know the mean, MAD and range of the demand distributions.   
When additional information is available, we can also construct a matching lower bound. We include the skewness information $\beta_i=\mathbb{P}(D_i \geq \mu_i)$ in the mean-MAD ambiguity set to obtain the tight lower bound. 
For the random variables $\vec{D} = (D_1,\dots,D_n)$, define the ambiguity set as
\begin{equation*}
    \mathcal{P}_{(\mu,\delta,\beta)} := \{\mathbb{P} \, | \, \mathbb{P} \in \mathcal{P}_{(\mu, \delta)} , \, \mathbb{P}(D_i \geq \mu_i) = \beta_i, \, i= 1,\dots,n \}
\end{equation*}
with $\mathcal{P}_{(\mu,\delta,\beta)}\subseteq\mathcal{P}_{(\mu,\delta)}$. The proof of the following result is identical to that of Lemma~\ref{lemma:meanmad}, but now uses the tight lower bound for a convex function of random variables discussed in \citet{ben1972more}.
To make this paper self-contained, a proof for the univariate case is provided in \ref{ec:proofs}. This is sufficient since the univariate result can be applied to each term of the summation in $G(\vec{q},\vec{D})$ separately, as with Lemma~\ref{lemma:meanmad}.

\begin{lemma}\label{lemma:meanmadbeta} 
The extremal distribution that solves
$\min\limits_{\mathbb{P} \in \mathcal{P}_{(\mu,\delta,\beta)}} \expectp{G(\vecc{q},\vec{D})}$
consists for each $D_i$ of a two-point distribution with values $\mu_i + \frac{\delta_i}{2\beta_i},\,  \mu_i - \frac{\delta_i}{2(1-\beta_i)}$ and probabilities
$\beta_i,\,1-\beta_i,$ respectively.
\end{lemma}

Using this result, we obtain
\begin{equation}\label{multiLowerBound}
\begin{aligned}
    \min_{\vecc{q}}& \quad C^L(\vecc{q}) :=\sum_{i=1}^n c_i\left(d_i(q_i-\mu_i) + (m_i+d_i)\left(\beta_i(\mu_i + \dfrac{\delta_i}{2\beta_i}-q_i)^+ + (1- \beta_i)(\mu_i-\dfrac{\delta_i}{2(1-\beta_i)}-q_i)^+  \right) \right)       \\
    \text{s.t }& \quad \sum_{i=1}^n c_i q_i \leq B, \\
    & \quad q_i \geq0, \quad \text{for } i = 1,\dots,n,  
\end{aligned}
\end{equation}
as a model to provide a lower bound for the multi-item newsvendor. As the objective function in problem \eqref{multiLowerBound} also consists of piecewise linear functions, there exists an LP representation and knapsack algorithm for \eqref{multiLowerBound} similar to the results for problem \eqref{model:ConstrainedMultiMeanMAD}. 

We can now solve \eqref{model:ConstrainedLPMultiMeanMAD} and \eqref{multiLowerBound} to obtain tight performance intervals for the multi-item newsvendor model, using recent DRO results (see \ref{ec:recentDRO} and \citealp{postek2018robust}). For all feasible ordering policies $\vecc{q}$ and $\mathbb{P}\in\mathcal{P}_{(\mu,\delta,\beta)}$, it holds that
\begin{equation*}
    C(\vecc{q}) \in \left[C^L(\vecc{q}), C^U(\vecc{q})\right].
\end{equation*}
In addition, for the optimal solutions to the newsvendor problem and its distributionally robust counterparts, 
\begin{equation*}
    C(\vecc{q}^*) \in \left[C^L(\vecc{q}^L), C^U(\vecc{q}^U)\right].
\end{equation*} 
One can find the tightest upper and lower bounds, based on mean-MAD ambiguity, for the multi-item newsvendor model by calculating the optimal solutions to models \eqref{model:ConstrainedMultiMeanMAD} and \eqref{multiLowerBound}, respectively.

\section{Numerical examples of robust ordering}\label{sec:numm}
We will now illustrate and visualize the robust ordering policies.
To demonstrate the  `budget-consistency' property, Section~\ref{subsecinc} applies the knapsack algorithm for a setting where the budget is increased. In Section~\ref{sec:performance} we
contrast the performance of the knapsack policy for partial demand information against that of the optimal solution for the full information setting.
Our code is made available in the form of an online supplement.

\subsection{Numerical illustration of the `budget-consistency' property}\label{subsecinc}

We illustrate the knapsack algorithm and the process of allocating budget to different order quantities for items in the newsvendor model. Consider $n=5$ identically distributed items with support $a=10,\ b=50$ and mean $\mu = 30$. From Figure~\ref{allocAlgorithmEx}, we can infer that item 1 is the most profitable. Low budget levels are allocated to this item such that we obtain $q_1 = \mu$. Item number 3 is the last item to which the budget is allocated. Hence, it is the least profitable item. Table~\ref{exampleG} displays the ordered set $\mathcal{G}$.
 From this table, we can indeed infer that item 1 has the smallest value for $\alpha_{i,0}/c_{i}$ and therefore is increased first. 
\begin{figure}[h!]
    \centering
    \includegraphics[width=0.6\textwidth]{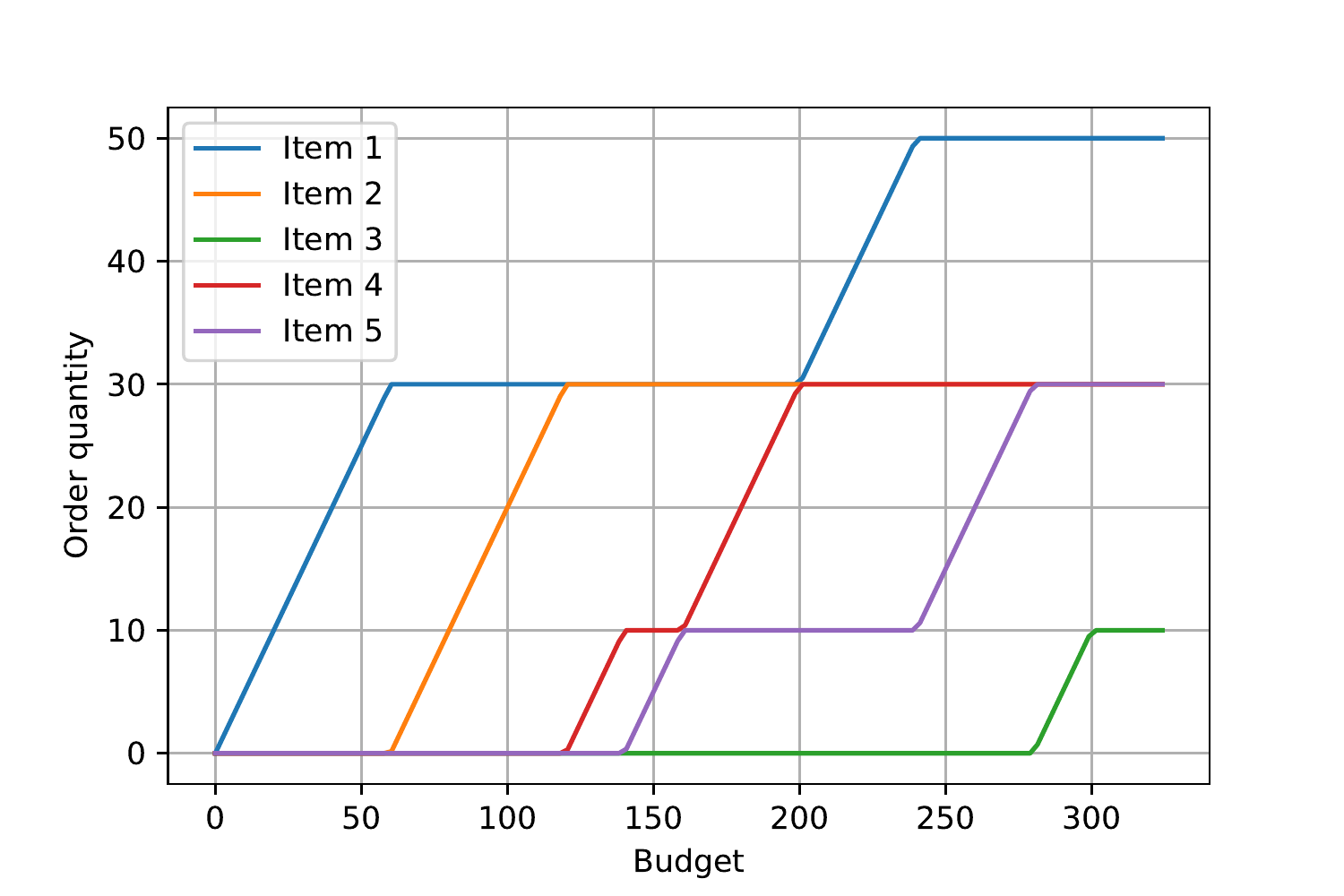}
    \caption{Development of the order quantities when the  budget increases according to the knapsack algorithm}
    \label{allocAlgorithmEx}
\end{figure}

\begin{table}[h!]
\centering
\caption{Table containing $\alpha_{i,j}/c_{i}$ and corresponding information of the ordered set $\mathcal{G}$}
\label{exampleG}
\scalebox{0.9}{
\begin{tabular}{c|ccccccccccccccc}
$\mathcal{G}$              & 1     & 2     & 3     & 4     & 5    & 6     & 7    & 8     & 9     & 10 & 11   & 12   & 13   & 14  & 15  \\ \hline
$\alpha_{i,j}/c_{i}$              & -0.92 & -0.75 & -0.72 & -0.49 & -0.3 & -0.15 & -0.1 & -0.08 & -0.03 & -0.01  & 0.14 & 0.42 & 0.45 & 0.7 & 0.7 \\
Function piece & 0     & 1     & 0     & 1     & 0    & 0     & 1    & 2     & 1     & 0  & 1    & 2    & 2    & 2   & 2   \\
Item           & 1     & 1     & 2     & 2     & 4    & 5     & 4    & 1     & 5     & 3  & 3    & 5    & 2    & 4   & 3  
\end{tabular}}
\end{table}

Figure~\ref{allocAlgorithmEx} nicely illustrates that when the budget is increased, the orders for the original budget remain
unaltered, while only the additional budget is further divided over the items. 
To further illustrate the `budget-consistency' property, consider the multi-item newsvendor model for which $n=2$, $m_2=2$, the remaining cost parameters equal 1, and demand is identically distributed according to a symmetric triangle distribution supported on $[10,50]$. In Figure~\ref{meanVarianceIllustration} we plot the expected costs and order quantities for various budget levels. Figure~\ref{meanVarianceIllustrationPolicy} contains the allocation between both order quantities. 
For low budget values, one first increases the order quantity of item one, the most profitable item. Figure~\ref{meanVarianceIllustrationBounds} shows the upper bound \eqref{model:ConstrainedMultiMeanMAD} and lower bound \eqref{multiLowerBound} 
that together lead to a tight performance interval for the expected costs.

For the sake of comparison, we also show results for the partial demand information setting considered in \citet{gallego1993distribution}, assuming that the mean and variance of demands are known; see  
\ref{sec:scarfgallegomoon} for more details. The results of  \citet{gallego1993distribution} depend (non-trivially) on all model parameters, including the budget $B$. This lack of budget-consistency forces the decision maker to solve an optimization problem, see \eqref{model:multiItemMeanVariance}, for each budget level separately, and explains the smooth curve in Figure~\ref{meanVarianceIllustrationPolicy}. In contrast, our knapsack algorithm generates a sorted ordering list that does not depend on $B$, and prescribes to sort items successively according to that list, with order sizes equal to the minimal, mean or maximum demand.

\begin{figure}[h!]
    \centering
    \begin{subfigure}{0.49\linewidth}
    \includegraphics[width=\textwidth]{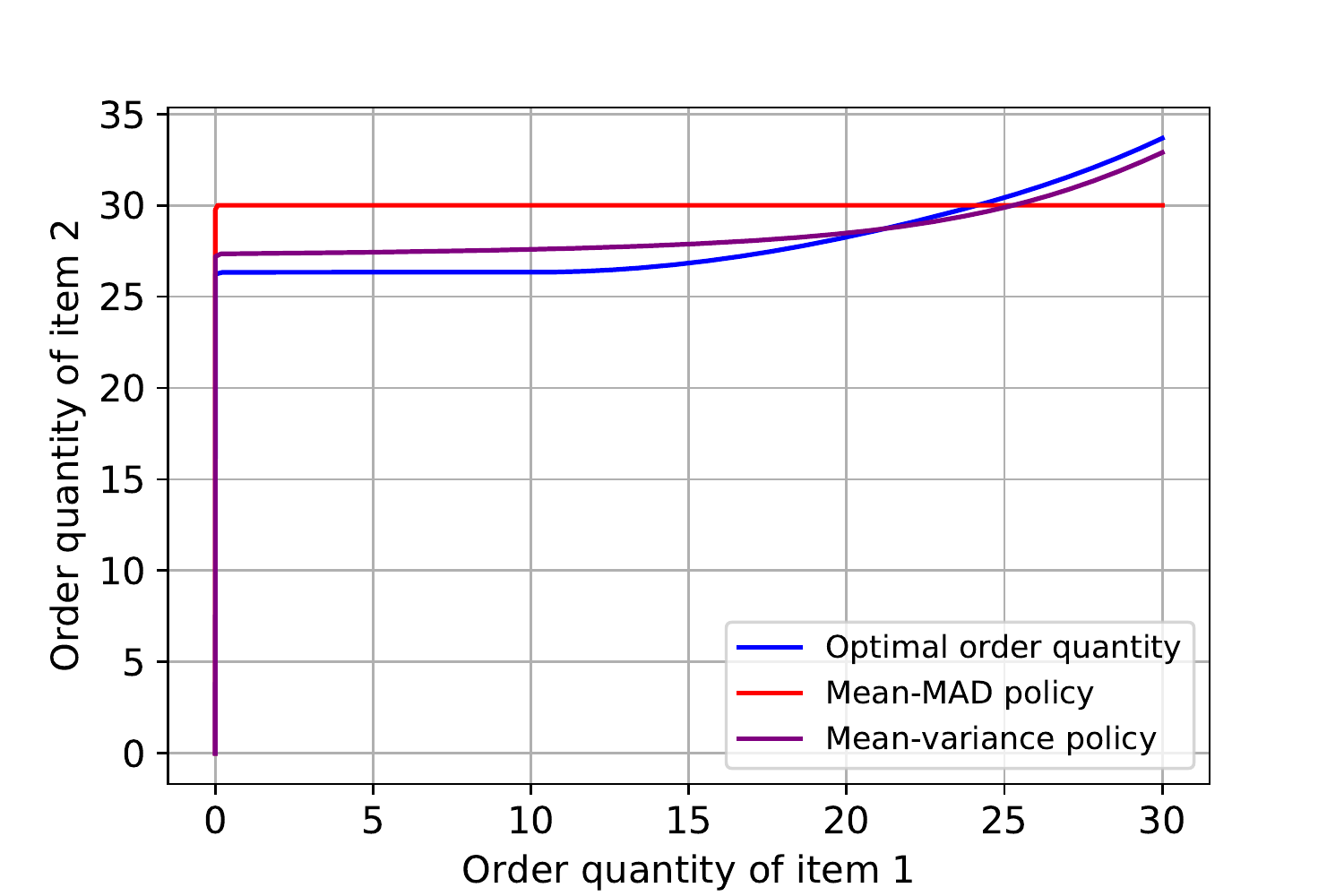}
    \caption{Ordering policy}
    \label{meanVarianceIllustrationPolicy}
    \end{subfigure}
    \hfill
    \begin{subfigure}{0.49\linewidth}
        \includegraphics[width=\textwidth]{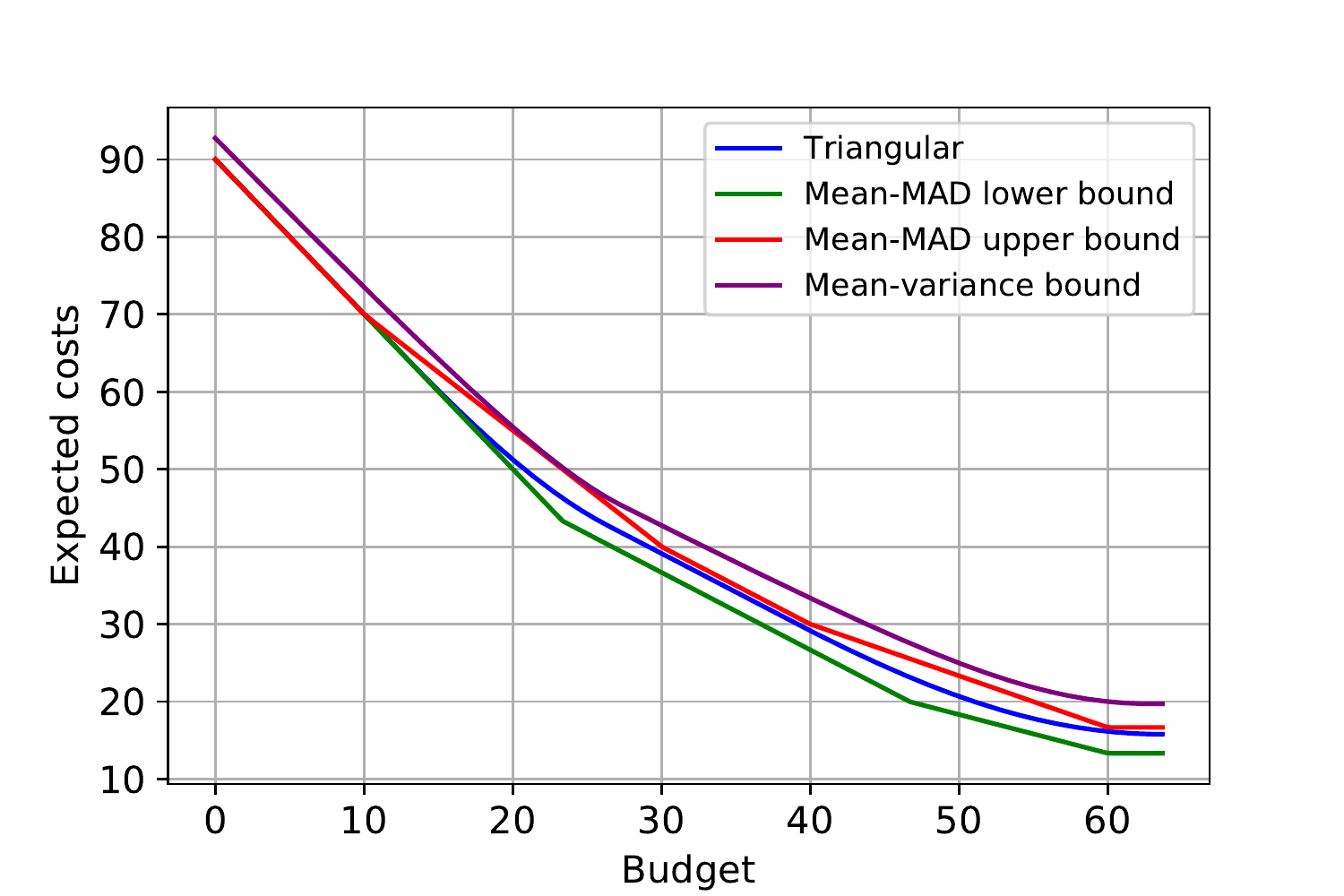}
    \caption{Newsvendor costs}
    \label{meanVarianceIllustrationBounds}
    \end{subfigure}
    \caption{Mean-variance and mean-MAD bounds and ordering policies for the newsvendor model. The mean-variance curves are obtained through solving \eqref{model:multiItemMeanVariance}. The mean-MAD policy corresponds to the optimal solution of \eqref{model:ConstrainedMultiMeanMAD}. The mean-MAD upper and lower bounds correspond to the extremal three- and two-point distributions, respectively. The ‘true’ cost function assumes that $D$ follows a symmetric triangular distribution on $[10,50]$.}\label{meanVarianceIllustration}
\end{figure}

We emphasize that these results are not meant to numerically compare the mean-MAD and mean-variance policies, because the displayed differences merely express different ways of dealing with ambiguity.  Indeed, it is hard to compare both policies as the respective ambiguity sets can contain vastly different distributions. For instance, a finite variance excludes distributions with an infinite second moment, while finite MAD does not. 
For our purposes, MAD and variance are equally adequate descriptors of dispersion, and both are easily calibrated on data using basic statistical estimators. The crucial difference in the DRO context of this paper is that MAD leads to a simple, budget-consistent ordering policy.


\subsection{Expected value of additional information}\label{sec:performance}
We introduce as performance measure the expected value of additional information (EVAI), defined as
\begin{equation*}
    \text{EVAI}(\vecc{q}^U_B) = \dfrac{C(\vecc{q}^U_B)-C(\vecc{q}^*_B)}{C(\vecc{q}^*_B)}, 
    \label{evaiMultiItemPerformance}
\end{equation*} 
where $\vecc{q}^U_B$ is the robust ordering policy and $\vecc{q}^*_B$ is the optimal ordering policy when the joint demand distribution is known. 
We let $B$ run from $0$ to $\sum_{i=1}^n q^*_i \eqqcolon B_{\text{opt}}$, and consider nine different demand distributions, listed in Table~\ref{tableDistributionsMultiItemPerformance}.

\begin{table}[h!]
\caption{Nine distributions used for multi-item performance analysis }
    \label{tableDistributionsMultiItemPerformance}
\centering
\begin{tabular}{cccccc}
\textbf{Case} & \ & \textbf{Case} & & \textbf{Case} &  \\ \hline
1             & $\text{Uniform}[10,50]$     & 4             & $\text{Beta}(1,3)$ on $[0,50]$ & 7             & $\text{Triangular}(10,50,18)$              \\ 
2             & $\text{Uniform}[10,100]$    & 5             & $\text{Beta}(2,2)$ on $[0,50]$ & 8             & $\text{Triangular}(10,50,30)$              \\
3             & $\text{Uniform}[10,200]$   & 6             & $\text{Beta}(3,1)$ on $[0,50]$ & 9             & $\text{Triangular}(10,50,42)$            \\ \hline
\end{tabular}
\end{table}


We consider $n=25$ items. For each item $i$, let $c_i=d_i=1$ and assume identically distributed demand. For example, in Case 2 the demand $D_i$ for each item $i$ follows the uniform distribution with parameters $a_i=10$ and $b_i=100$.
Table \ref{tab4a} provides an overview for the mark-up, representing low, average and high margins. 

\begin{table}[h!]
\centering
\caption{Mark-up values for all 25 items in the newsvendor model}
\label{tab4a}
\begin{tabular}{c|ccccccccccccc}
\hline
Mark-up        & \multicolumn{1}{l}{$m_1$} & \multicolumn{1}{l}{$m_2$} & \multicolumn{1}{l}{$m_3$} & \multicolumn{1}{l}{$m_4$} & \multicolumn{1}{l}{$m_5$} & \multicolumn{1}{l}{$m_6$} & \multicolumn{1}{l}{$m_7$} & \multicolumn{1}{l}{$m_8$} & \multicolumn{1}{l}{$m_9$} & \multicolumn{1}{l}{$m_{10}$} & \multicolumn{1}{l}{$m_{11}$} & \multicolumn{1}{l}{$m_{12}$} & \multicolumn{1}{l}{$m_{13}$}  \\ \hline
Low margin      & 0.1                       & 0.14                      & 0.18                      & 0.21                      & 0.25                      & 0.29                      & 0.33                      & 0.36                      & 0.4                       & 0.44         & 0.48     & 0.51     & 0.55               \\
Average margin  & 1                         & 1.13                      & 1.25                      & 1.38                      & 1.5                       & 1.63                      & 1.75                      & 1.88                      & 2                         & 2.13          & 2.25     & 2.38     & 2.5             \\
High margin    & 4                         & 4.21                      & 4.42                      & 4.63                      & 4.83                      & 5.04                      & 5.25                      & 5.46                      & 5.67                      & 5.88      & 6.08     & 6.29     & 6.5        \\ \hline         

Mark-up                 & $m_{14}$ & $m_{15}$ & $m_{16}$ & $m_{17}$ & $m_{18}$ & $m_{19}$ & $m_{20}$ & $m_{21}$ & $m_{22}$ & $m_{23}$ & $m_{24}$ & $m_{25}$ \\ \hline
Low margin      & 0.59     & 0.63     & 0.66     & 0.7      & 0.74     & 0.78     & 0.81     & 0.85     & 0.89     & 0.93     & 0.96     & 1        \\
Average margin  & 2.63     & 2.75     & 2.88     & 3        & 3.13     & 3.25     & 3.38     & 3.5      & 3.63     & 3.75     & 3.88     & 4        \\
High margin     & 6.71     & 6.92     & 7.12     & 7.33     & 7.54     & 7.75     & 7.96     & 8.17     & 8.37     & 8.58     & 8.79     & 9 \\\hline      
\end{tabular}
\end{table}

For the low margin regime, Figure~\ref{performanceMADlow} shows results for each of the nine cases, for both the robust ordering policy with mean-MAD-range information, and for the policy that uses the additional information $\beta_i=\P(D_i\geq\mu_i)$.
For the former, the worst performance over all nine cases has a maximum deviation of approximately 23\% compared to the optimal order quantity $q^*_B$. Overall, the performance of the robust policy only deviates a few percent from the optimal performance with full information availability. For the uniformly distributed cases (Cases~1-3), the performance decreases when the range increases. For beta distributed demand (Cases~4-6), right-tailed distributions perform worse than left-tailed distributions. This effect is also observed for the triangular distributions (Cases~7-9). The  policy with additional information $\beta_i=\P(D_i\geq\mu_i)$ performs somewhat better in most cases.

Figure \ref{performanceMADhigh} shows similar results for  high margins. 
The EVAI for the robust policy remains mostly below 10\% for lower budget levels, but starts increasing rapidly when the budget approaches $B_{\text{opt}}$ (i.e., when approaching the unconstrained model). When the budget is less restrictive, additional distributional information provides substantial value. In particular, since the  policy uses skewness information $\beta_i$, it performs better (in expectation) for higher budget levels than the robust ordering policy. We present some more performance plots for the average margin setting and additional numerical experiments with mean-variance information in \ref{ec:moreresults}.

\begin{figure}[h!]
    \centering
    \includegraphics[width=.85\linewidth]{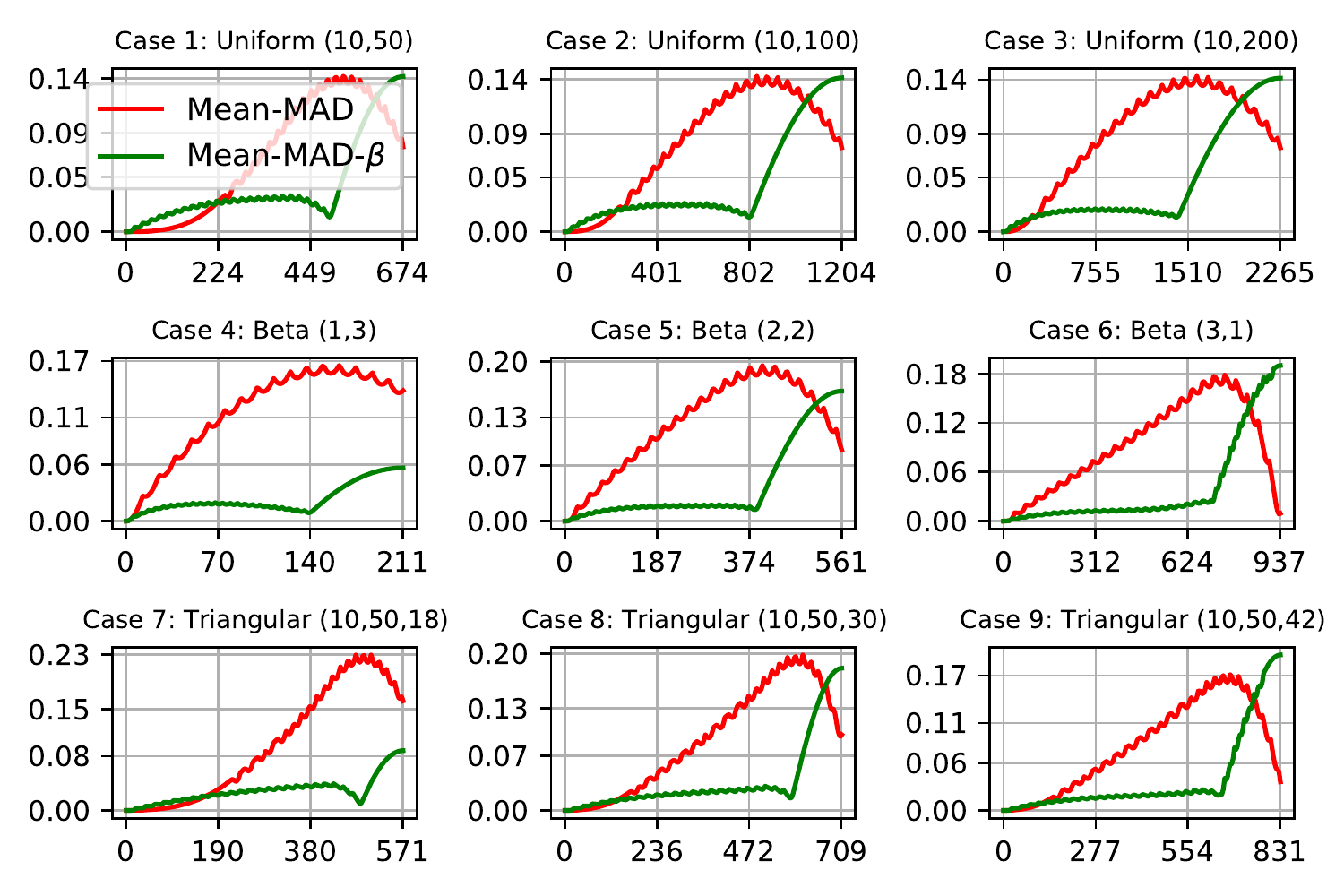}
    \caption{The results for the low margin setting. The x-axis corresponds to $B$ and the y-axis to the EVAI.}
    \label{performanceMADlow}
\end{figure}

\begin{figure}[h!]
    \centering
    \includegraphics[width=.85\linewidth]{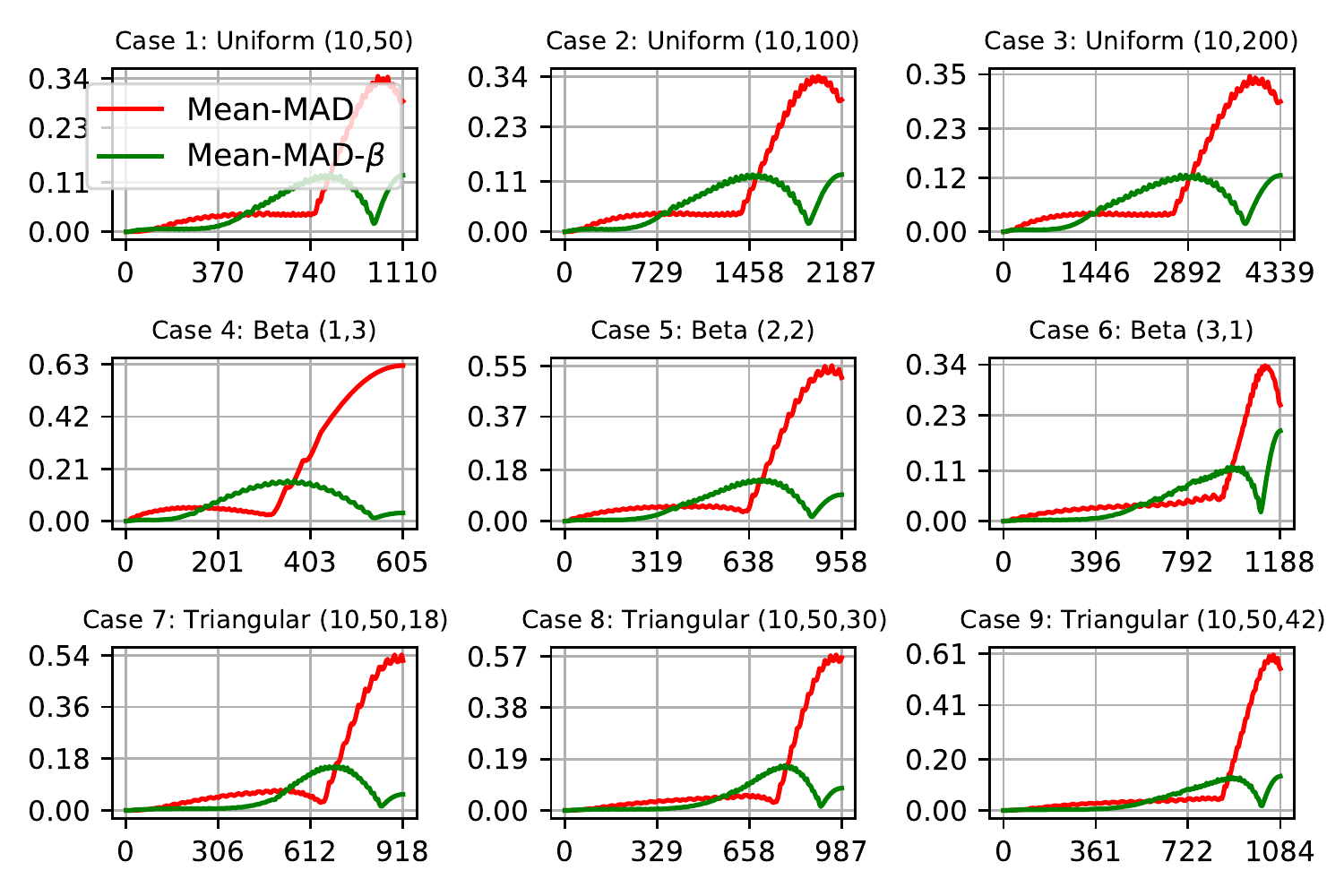}
    \caption{The results for the high margin setting. The x-axis corresponds to $B$ and the y-axis to the EVAI.}
    \label{performanceMADhigh}
\end{figure}

We next quantify the value of MAD information by comparing the performance with the situations when only the mean and range of demand is known. For the low margin setting, Figure~\ref{performanceEM} shows the EVAI for the ordering policy with only mean-range information. Like the mean-MAD policy, this policy follows from a discrete distribution, in this case the extremal distribution on $\{a,b\}$ with probabilities $\frac{b-\mu}{b-a}$ and $\frac{\mu-a}{b-a}$ that attains the Edmundson-Madansky bound (see \citealp{ben1972more}). That is, instead of the worst-case three-point distribution, we take the expectation in \eqref{model:CapNewsModel} over this two-point distribution and find the robust mean-range ordering policy using the resulting LP. The plots clearly demonstrate that knowledge on dispersion in terms of MAD improves performance considerably.

\begin{figure}[h]
    \centering
    \includegraphics[width=.85\linewidth]{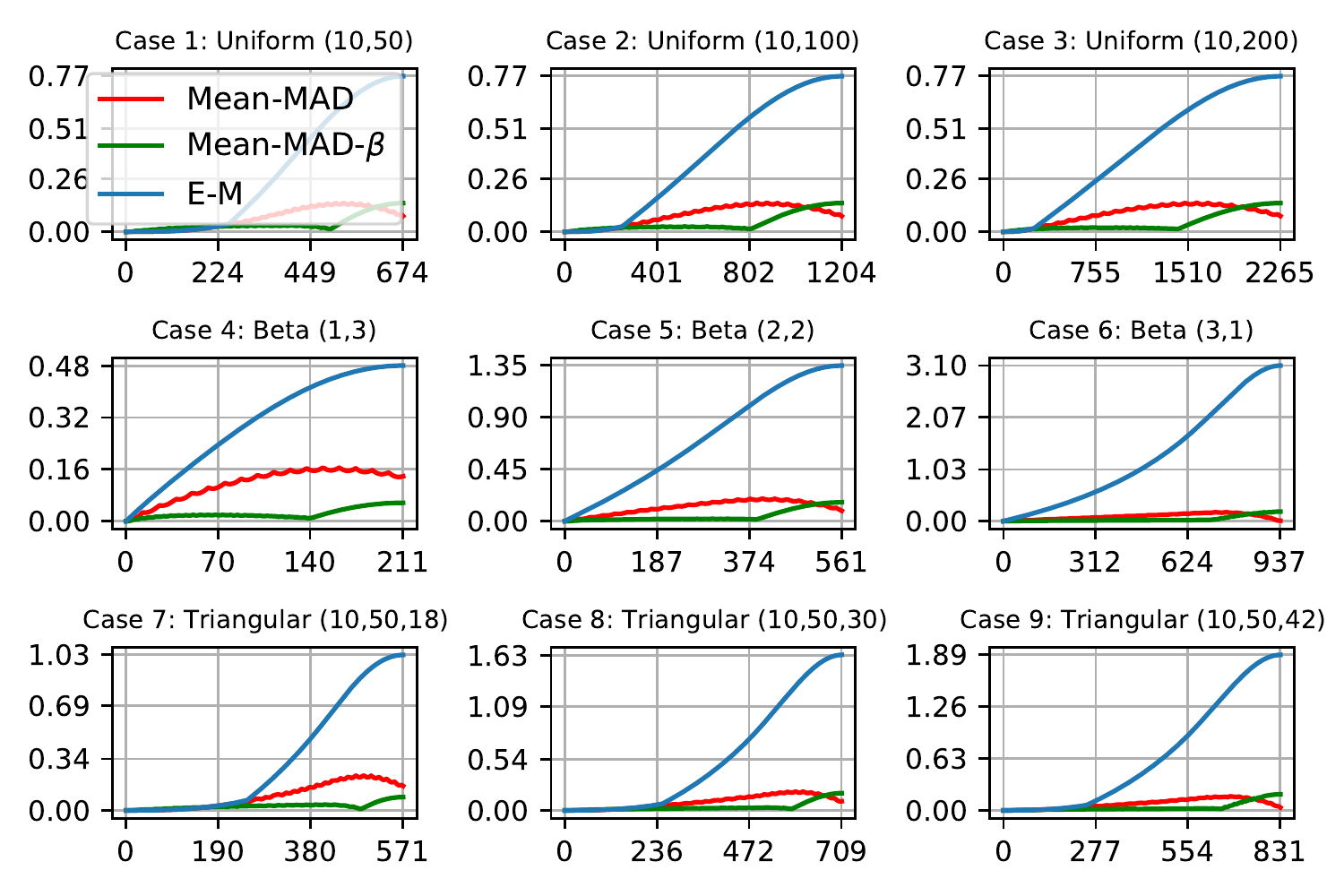}
    \caption{The results for the low margin setting. The x-axis corresponds to $B$ and the y-axis to the EVAI. The E-M performance plot refers to the model with only mean information.}
    \label{performanceEM}
\end{figure}


\newpage

\section{Conclusions}\label{sec:outlook}
This paper establishes new ordering policies for the newsvendor with partial demand information (mean, MAD and range) with a budget constraint. 
The ordering policies follow from a minimax approach, where we search for the order quantities with minimal costs for the maximal (worst-case) cost function 
restricted to demand distributions that comply with the partial information. 

The minimax analysis for the multi-item setting gives rise to a knapsack problem, and the solution of this knapsack problem in fact is the ordering policy. This policy prescribes to sort items based on their marginal effect on the total costs, reminiscent of the greedy algorithm that solves the continuous knapsack problem. The ordering policy only orders the minimum, mean or maximum demand for each item. Hence, the decision maker can rank the items based on their marginal effects, and then start ordering items according to this list until the budget is spent. The fact that the ranking list is easy to generate, and that the `order of ordering' does not depend on the budget, makes the policy transparent and easy to implement. Existing approaches for full and partial (such as mean-variance) knowledge of the demand distribution lack this property of `budget-consistency'.


The minimax approach provides robustness, with an ordering policy that protects against all distributions that comply with the partial information. 
This approach avoids the need to estimate the demand distribution, which can be a daunting process in practice and is prone to errors.
However, the minimax approach comes at the risk of being overly conservative. 
Through extensive numerical experiments we compared the robust policies for partial demand settings with the policies for full demand settings, and observed that the proposed policies perform well.

At the heart of our analysis lies the idea to set up the robust minimax analysis with MAD information. With MAD as dispersion measure we obtained a tractable optimization model, with a solution in terms of a robust ordering policy that satisfies the budget-consistency property. 
Using MAD to formulate solvable minimax problems can also be applied to other inventory models. We demonstrate this idea in \ref{sec:4} for three extended settings: the newsvendor with multiple contraints, the newsvendor with unreliable supply, and the risk-averse newsvendor. In all three cases, the minimax analysis leads to a tractable mathematical program, either a knapsack problem or a linear program.

\bibliography{arXiv_v1}
\bibliographystyle{apalike}

\ECSwitch


\ECHead{E-Companion to ``Robust knapsack ordering for a partially-informed newsvendor with  budget constraint''}

\section{Proofs}\label{ec:proofs}


\begin{proof}{Proof of Lemma~\ref{lemma:meanmad1}}
In their original work, \cite{ben1972more} prove this result for general convex functions by dividing the support into two intervals $[a,\mu]$ and $[\mu,b]$ and then applying the Edmundson-Madansky bound to both subintervals. The following proof uses semi-infinite programming duality and is taken from \citet{vaneekelen2021mad}. Consider a general convex function $f(x)$ (this includes $(x-q)^+$ as a special case). 
For $X\sim\mathbb{P}\in\mathcal{P}_{(\mu,\delta)}$, we solve
\begin{equation}\label{test3}
\begin{aligned}
&\max_{\P(x)\geq0} &  &\int_a^b f(x){\rm d} \P(x)\\
&\text{s.t.} &      & \int_a^b {\rm d}\P(x)=1,\  \int_a^b x\,{\rm d}\P(x)=\mu,\ \int_a^b |x-\mu|{\rm d}\P(x)=\delta,   
\end{aligned}
\end{equation}

Consider the dual of \eqref{test3},
\begin{equation}\label{test4}
\begin{aligned}
&\min_{\lambda_0,\lambda_1, \lambda_2} &  &\lambda_0 +\lambda_1 \mu+\lambda_2 \delta\\
&\text{s.t.} &      & M(x):=\lambda_0 +\lambda_1 x + \lambda_2|x-\mu| \geq f(x), \ \forall x\in[a,b].
\end{aligned}
\end{equation}
The function $M(x)$ has a `kink' at $x=\mu$. Since the dual problem \eqref{test4} has three variables, the optimal $M(x)$ touches $f(x)$ at three points: $x=a$, $\mu$ and $b$. For this choice of $M(x)$, 
\begin{align*}
\lambda_0&=f(a)-\lambda_1 a - \lambda_2(\mu-a),\ \lambda_1=\frac12\left(\frac{f(b)-f(\mu)}{b-\mu}+\frac{f(\mu)-f(a)}{\mu-a}\right),\\
\lambda_2&=\frac12\left(\frac{f(b)-f(\mu)}{b-\mu}-\frac{f(\mu)-f(a)}{\mu-a}\right). 
\end{align*}
Because the majorant is piecewise linear and convex, we can majorize every convex function $f(x)$ by letting $M(x)$ touch at the boundary points $a,b$ and at the kink point $x=\mu$. According to the complementary slackness property, these points constitute the support of the extremal distribution, and the optimal probabilities follow from solving the linear system resulting from the equations of \eqref{test3}. This is a linear system of three unknown probabilities and three equations, with the solution
$$
p_{a}  = \frac{\delta}{2(\mu - a)}, \quad p_{\mu}  = 1- \frac{\delta}{2(\mu - a)}  - \frac{\delta}{2(b - \mu)} , \quad p_{b}  = \frac{\delta}{2(b - \mu)}.
$$
Finally, for these primal and dual solutions, we verify that the objective values of problems \eqref{test3} and \eqref{test4} agree, which confirms that strong duality holds.
\end{proof}

\begin{proof}{Proof of Lemma~\ref{lemma:meanmadbeta}}
We prove this result for general convex $f(x)$. For a random variable $X$ with distribution $\mathbb{P}\in\mathcal{P}_{(\mu,d,\beta)}$, the tight lower bound follows from
\begin{equation}\label{eq:primallb}
\begin{aligned}
&\max_{\P(x)\geq0} &  &\int_a^b f(x){\rm d} \P(x)\\
&\text{s.t.} &      & \int_a^b {\rm d}\P(x)=1,\  \int_a^b x\,{\rm d}\P(x)=\mu,\ \int_a^b |x-\mu|{\rm d}\P(x)=\delta,\  \int_a^b \mathbbm{1}_{\{x\geq\mu\}}{\rm d}\P(x)=\beta.  
\end{aligned}
\end{equation}

Consider the dual of \eqref{eq:primallb},
\begin{equation}\label{eq:duallb}
\begin{aligned}
&\min_{\lambda_0,\lambda_1, \lambda_2} &  &\lambda_0 +\lambda_1 \mu+\lambda_2 \delta + \lambda_3 \beta\\
&\text{s.t.} &      & M(x):=\lambda_0 +\lambda_1 x + \lambda_2|x-\mu| + \lambda_3 \mathbbm{1}_{\{x\geq\mu\}} \leq f(x), \ \forall x\in[a,b].
\end{aligned}
\end{equation}
Here $M(x)$ has both a ‘kink’ and a jump discontinuity at $x=\mu$. Let the function $M(x)$ touch the epigraph of $f(x)$ in two points on opposite sides of $\mu$. If we insert this knowledge, the constraints in the dual problem reduce to two equality constraints.
From the Karush-Kuhn-Tucker conditions, we deduce the optimal tangent points:
$$
x_1=\mu+\frac{\delta}{2\beta}, \quad x_2=\mu-\frac{\delta}{2(1-\beta)},
$$
which correspond to $\upsilon_1$ and $\upsilon_2$. Substituting this solution and solving for $\lambda_0, \lambda_1, \lambda_2$ and $\lambda_3$ gives
$$
\lambda_0=f(\upsilon_2)+\frac{(\lambda_1 - \lambda_2)\delta}{2(1-\beta)}-\lambda_1\mu, \quad \lambda_3=f(\upsilon_1)-f(\upsilon_2)+\frac{\lambda_2\delta}{(1-\beta)}-\frac{(\lambda_2+\lambda_1)\delta}{2\beta(1-\beta)},
$$
and hence the optimal value is given by $\beta f(\upsilon_1)+(1-\beta)f(\upsilon_2)$. To ensure the solution is dual feasible, we assign suitable values to the two free decision variables. That is, we let $\lambda_1+\lambda_2$ and $\lambda_1-\lambda_2$ equal the slope of $f(x)$ at $x=\upsilon_1$ and $\upsilon_2$, respectively. The optimal probabilities of \eqref{eq:primallb} are obtained by solving the linear system resulting from \eqref{eq:primallb}.
\end{proof}


\section{Known properties of MAD}\label{ec:propertiesMAD}
We recall some well-known properties of the MAD; see e.g. \cite{ben1985approximation}. Denote by $\sigma^2$ the variance of the random variable $X$, whose distribution is known to belong to the set $\mathcal{P}_{(\mu,\delta)}$. Then
$$
\frac{\delta^2}{4\beta (1-\beta)} \leq \sigma^2 \leq \frac{\delta(b-a)}{2}.
$$
In particular, since
$$
\delta^2 \leq 4\beta (1-\beta) \sigma^2 \leq \sigma^2,
$$
it holds that $\delta \leq \sigma$. For a proof, we refer the reader to \cite{ben1985approximation}. For the distributions used in the paper, explicit formulas for $\delta$ are available:
\begin{itemize}
\item Uniform distribution on $[a,b]$:
$$
\delta = \frac{1}{4}(b-a)
$$
\item Beta distribution with parameters $k,\lambda$ on support $[a,b]$:
$$
\delta = \frac{2k^{k}\lambda^{\lambda}\Gamma(k+\lambda)}{(k+\lambda)^{k+\lambda+1}\Gamma(k)\Gamma(\lambda)}(b-a)
$$
\item Triangular distribution on $[a,b]$ with mode $c$:
$$
\begin{aligned}
\delta = \begin{cases}
\frac{2(b+c-2a)^3}{81(a-b)(a-c)}, \quad & \text{for } a+b<2c, \\
\frac{2(a+c-2b)^3}{81(a-b)(b-c)}, &       \text{for } a+b>2c    
\end{cases}
\end{aligned}
$$
\item Normal distribution $N(\mu,\sigma^2)$:
$$
\delta = \sqrt{\frac{2}{\pi}} \sigma
$$
\item Gamma distribution with parameters $\lambda$ and $k$ (for which $\mu = k/ \lambda$):
$$
\delta = \frac{2k^k}{\Gamma(k)\exp(k)} \frac{1}{\lambda}.
$$
\end{itemize}
The MAD is known to satisfy the bound
\begin{equation}
\label{eq:d_bound}
0 \leq \delta\leq \frac{2(b-\mu)(\mu-a)}{b-a}.
\end{equation}
Let $\beta=\mathbb{P}(X \geq \mu)$. For example, in the case of continuous symmetric distribution of $X$ we know that $\beta = 0.5$. This quantity is known to satisfy the bounds:
\begin{equation}
\label{eq:beta_bound}
\frac{\delta}{2(b-\mu)} \leq \beta \leq  1-\frac{\delta}{2(\mu - a)}.
\end{equation}

\section{The knapsack problem}\label{ec:Knapsack}
The knapsack problem \citep{kellerer2004multiple} is an integer programming problem and can be formulated as 
\begin{equation}\label{appendixModel}
\begin{aligned}
    \max_x& \quad  \sum_{i=1}p_ix_i \\
    \text{s.t. }& \quad \sum_{i=1}^n c_ix_i \leq B, \\
    &x_i \in \{0,1\}, \quad 1=1,\dots,n.
\end{aligned}
\end{equation}
for decision variable $x$, budget $B$, price $p>0$ and costs $c$. Assume $B < \sum_{i=1}^n c_i$. The continuous version is obtained by considering the linear relaxation, i.e., we replace the integrality constraints by $0 \leq x_i \leq 1, \, i=1,\dots,n.$  The so-called greedy choice algorithm  produces an optimal solution for the continuous knapsack problem. 

We first renumber the items $x_i$ such that ${p_1}/{c_1} \geq \dots \geq  {p_n}/{c_n}$. Hence, the first item causes the largest increase in value relative to its costs. We now iterate over $x_1,\dots,x_n$ and in each iteration, set $x_i$ to its maximum capacity. When the budget constraint is violated, set
$$ x_i = B - \sum_{i=1}^{i-1} c_ix_i. $$
This greedy choice algorithm produces the optimal solution to \eqref{appendixModel}. Below we will state its proof, which is an adaptation from the proof in \citet{kellerer2004multiple}. 

Assume that without loss of generality that $ {p_1}/{c_1} > \dots >  {p_n}/{c_n}$. If we would have $ {p_i}/{c_i} =  {p_{i+1}}/{c_{i+1}}$ for some $i$, then we are indifferent between those items and the proof below can be easily adapted to satisfy this. The greedy choice algorithm produces a solution such that, for some index $j$, we have $1=x_1=\dots = x_{j-1} > x_j \geq x_{j+1} = \dots = x_n = 0$. Suppose we would have a different feasible optimal solution $y \neq x$. Since $p_i > 0$ and $\sum_{i=1}^n c_i > B$, it must hold that $\sum_{i=1}^n c_iy_i = B$ as otherwise we could spend additional capital to increase the optimal value. Because $ {p_1}/{c_1} \geq \dots \geq  {p_n}/{c_n}$, there exists a smallest index $k$ such that $y_k <1$ and let $l$ be the smallest index such that $k < l $ and $y_l > 0 $. This solution must exists, else we would have $y=x$. Now, we will increase the value of $y_k$ and decrease the value of $y_l$. By choosing $\epsilon = \min\{c_k(1-y_k),c_ly_l\} > 0 $ and increasing $y_k$ by ${\epsilon}/{c_k}$ and decreasing $y_l$ by ${\epsilon}/{c_l}$, we maintain feasibility and preserve $\sum_{i=1}^nc_iy_i=B$. The solution value  changes by $p_k {\epsilon}/{c_k} - p_l{\epsilon}/{c_l} = \epsilon\left({p_k}/{c_k} - {p_l}/{c_l}\right) > 0 $. This contradicts the assumption that $y$ is an optimal solution. Therefore, $x$ is optimal which concludes the proof.

\section{DRO results}\label{ec:recentDRO}

In \citet{ben1972more}, the following result was proved (for a much larger class of functions $f(\vecc{y},\vec{X})$ than in our case):
\begin{proposition}\label{postekprop}
If $f(\vecc{y},\cdot)$ is convex,
\begin{equation}
\label{eq:convex_constraint_reformulated_general}
\sup\limits_{\mathbb{P} \in \mathcal{P}_{(\mu,\delta)}} \expectp{ f(\vecc{y}, \vec{X})} =  g_{\UU}(\vecc{y}) = \sum\limits_{\bm{\kappa} \in \{1,2,3 \}^{n}} \prod\limits_{i=1}^{n} p_{\kappa_i}^{(i)}  f(\vecc{y},\xi_{\kappa_1}^{(1)} ,\ldots , \xi_{\kappa_{n}}^{(n)}) ,
\end{equation}
with $p_{\kappa_i}^{(i)} , \xi_{\kappa_i}^{(i)} $ defined as in  {\rm Lemma~\ref{lemma:meanmad}}.
If $f(\vecc{y},\cdot)$ is concave,
\begin{equation}
\label{eq:concave_constraint_reformulated_general}
\sup\limits_{\mathbb{P} \in \mathcal{P}_{(\mu,\delta,\beta)}} \expectp{ f(\vecc{y}, \vec{X})} = g_L(\vecc{y}) = \sum\limits_{\bm{\kappa} \in \{1,2\}^{n}} \prod\limits_{i=1}^{n} \hat{p}_{\kappa_i}^{(i)} f( \vecc{y},\upsilon_{\kappa_1}^{(1)} ,\ldots,\upsilon_{\kappa_{n}}^{(n)} ),
\end{equation} 
with $\upsilon_1^{(i)}  = \mu_i + \frac{\delta_i}{2\beta_i},\, \upsilon_2^{(i)}  = \mu_i - \frac{\delta_i}{2(1-\beta_i)}$ and
$\hat{p}_1^{(i)}  = \beta_i,\, \hat{p}_2^{(i)}  = 1-\beta_i.$

\end{proposition}

Hence, $g_{\UU}(\cdot)$ in \eqref{eq:convex_constraint_reformulated_general} inherits the convexity in $\vecc{y}$ from $f(\cdot,\vec{X})$ and its functional form depends only on the form of $f(\cdot,\vec{X})$ (and similarly for $g_L(\cdot)$).
The upper and lower bound give a closed interval for
\begin{equation}
\label{eq:whatever}
\text{Val}_\mathbb{P}(\vecc{y}) = \expectp{ f(\vecc{y}, \vec{X})}  \quad \forall \mathbb{P} \in  \mathcal{P}_{(\mu,\delta,\beta)}.
\end{equation}
\begin{corollary}\label{corollary.Val.bounds}
If $f(\vecc{y},\cdot)$ is convex for all $\vecc{y}$ then
$
{\rm Val}_{\mathbb{P}}(\vecc{y}) \in [g_L(\vecc{y}),g_{\UU}(\vecc{y})] \ \forall \mathbb{P} \in \mathcal{P}_{(\mu,\delta,\beta)}.$
If $f(\vecc{y},\cdot)$ is concave for all $\vecc{y}$ then
$
{\rm Val}_{\mathbb{P}}(\vecc{y}) \in [g_{\UU}(\vecc{y}),g_L(\vecc{y})] \ \forall \mathbb{P} \in \mathcal{P}_{(\mu,\delta,\beta)}.$
\end{corollary}

From Proposition~\ref{postekprop} we see that the extremal distribution is independent of $\vecc{y}$. Hence, we can substitute the $3^n$ terms. This leads to a convex function in $\vecc{y}$, and hence the minimization problem over $\vecc{y}$ is tractable.

\section{Robust analysis with mean-variance knowledge}\label{sec:scarfgallegomoon}

\subsection{Scarf's result for single item}
 \citet{Scarf1958} introduced a distribution-free analysis for the single-item newsvendor model by assuming that the decision maker only knows the mean and variance of the demand. Define the ambiguity set containing all distributions with the same mean and variance as
$$    \mathcal{P}_{(\mu,\sigma)} := \{ \mathbb{P} \,|\, \mathbb{E}_{\mathbb{P}}(D) = \mu, \, \mathbb{E}_{\mathbb{P}}(D^2)= \sigma^2 + \mu^2 \}.$$
 \citet{Scarf1958} determined an upper bound on the cost function $C(q)$
 by finding the worst-case distribution in the ambiguity set. To find the order quantity that protects against the ambiguity in $\mathcal{P}_{(\mu,\sigma)}$, the following \textit{minimax} optimization problem is solved:
 $$
    \min_q \max_{\mathbb{P} \in \mathcal{P}_{(\mu,\sigma)}} dq + (m+d)\mathbb{E}_{\mathbb{P}}(D-q)^+.
 $$
Since
\begin{equation*}
\max_{\mathbb{P} \in  \mathcal{P}_{(\mu,\sigma)}}\mathbb{E}_{\mathbb{P}}(D-q)^+
\leq  \dfrac{\sqrt{\sigma^2 + (\mu-q)^2}+(\mu-q)}{2}, 
\end{equation*}
this {minimax} optimization problem becomes 
$ \min_q \max_{\mathbb{P}}  C^S(q)$ with 
\begin{equation}\label{eq:meanVarInequality}
 C^S(q) :=  d(q-\mu) + (m+d) \dfrac{\sqrt{\sigma^2 + (\mu-q)^2}+(\mu-q)}{2}. 
\end{equation}
and solution
\begin{equation}\label{eq:ScarfsRule}
  q^S := \argmin_q C^S(q)=  \mu + \dfrac{\sigma}{2} \left(\sqrt{\dfrac{m}{d}} - \sqrt{\dfrac{d}{m}} \right).
\end{equation}
The quantity $q^S$ is known as Scarf's order quantity which prescribes to order more than the expected demand  when $m > d$, and less than the expected demand when $d < m $. 

\subsection{Gallego and Moon}
When the model is based on mean-variance information, \citet{gallego1993distribution} formulate the problem as
\begin{align}
    \min_{\vecc{q}} C^{S}(\vecc{q}) := &\sum_{i=1}^{n} c_{i}\left(d_{i} (q_{i}-\mu_i)+\left(m_{i}+d_{i}\right) \frac{\sqrt{\sigma_{i}^{2}+\left(q_{i}-\mu_{i}\right)^{2}}-\left(q_{i}-\mu_{i}\right)}{2}\right) \nonumber
        \label{model:multiItemMeanVariance}\\ 
    \text{s.t.} \quad &\sum_{i=1}^n c_i q_i \leq B, \\
    &q \geq 0 \nonumber. 
\end{align}
The optimal solution to problem \eqref{model:multiItemMeanVariance} is referred to as $\vecc{q}^S$. Applying Scarf's bound for each item individually results in \eqref{model:multiItemMeanVariance}. Similar to the full information setting with a known distribution, this optimization problem can be solved with Lagrange multiplier techniques.

\section{Additional numerical experiments}\label{ec:moreresults}
This section presents additional numerical results. Section~\ref{ec:moremeanmad} presents the performance plots for the average margin setting. We compare the mean-MAD and mean-variance ordering policies in Section~\ref{ec:meanvarcomp}.

\begin{figure}[h!]
    \centering
    \includegraphics[width=.8\linewidth]{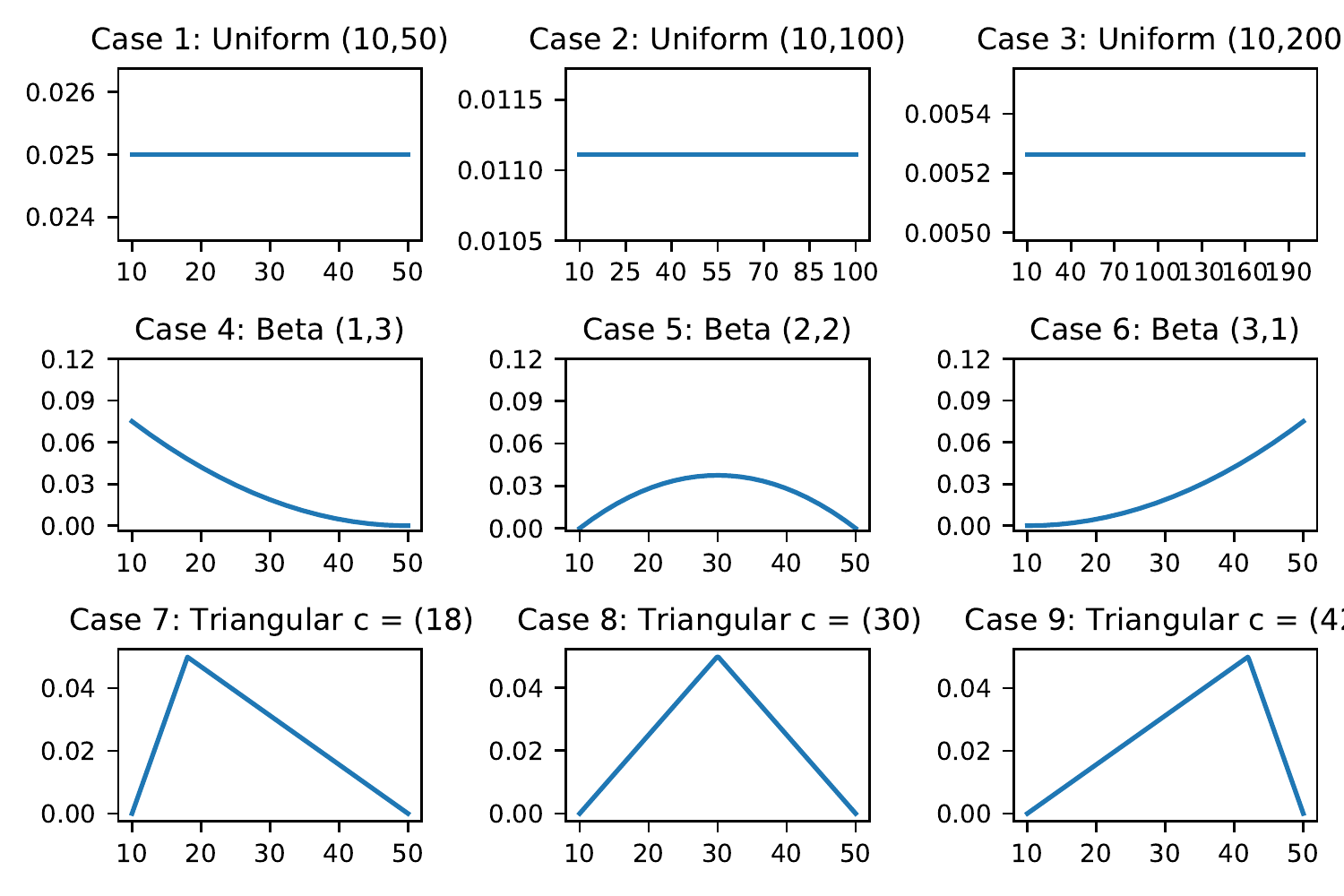}
    \caption{Nine probability density functions used for multi-item performance analysis}
    \label{densitySingleItemPerformance}
\end{figure}

\subsection{More mean-MAD results}\label{ec:moremeanmad}
Figure~\ref{performanceMADaverage} depicts the results for the average profitability scenario. A quick glance reveals that these plots exhibit a different impression than the low profitability scenario. We conclude that the mean-MAD EVAI remains below some bound for budget levels ranging from zero to two-thirds of the maximum budget. For all cases, this bound on the EVAI is around 10\%.
As the budget passes two-thirds of the maximum budget, the performance starts to decrease. However, the mean-MAD-$\beta$ EVAI decreases when approaching the maximal budget.

\begin{figure}[h]
    \centering
    \includegraphics[width=.85\linewidth]{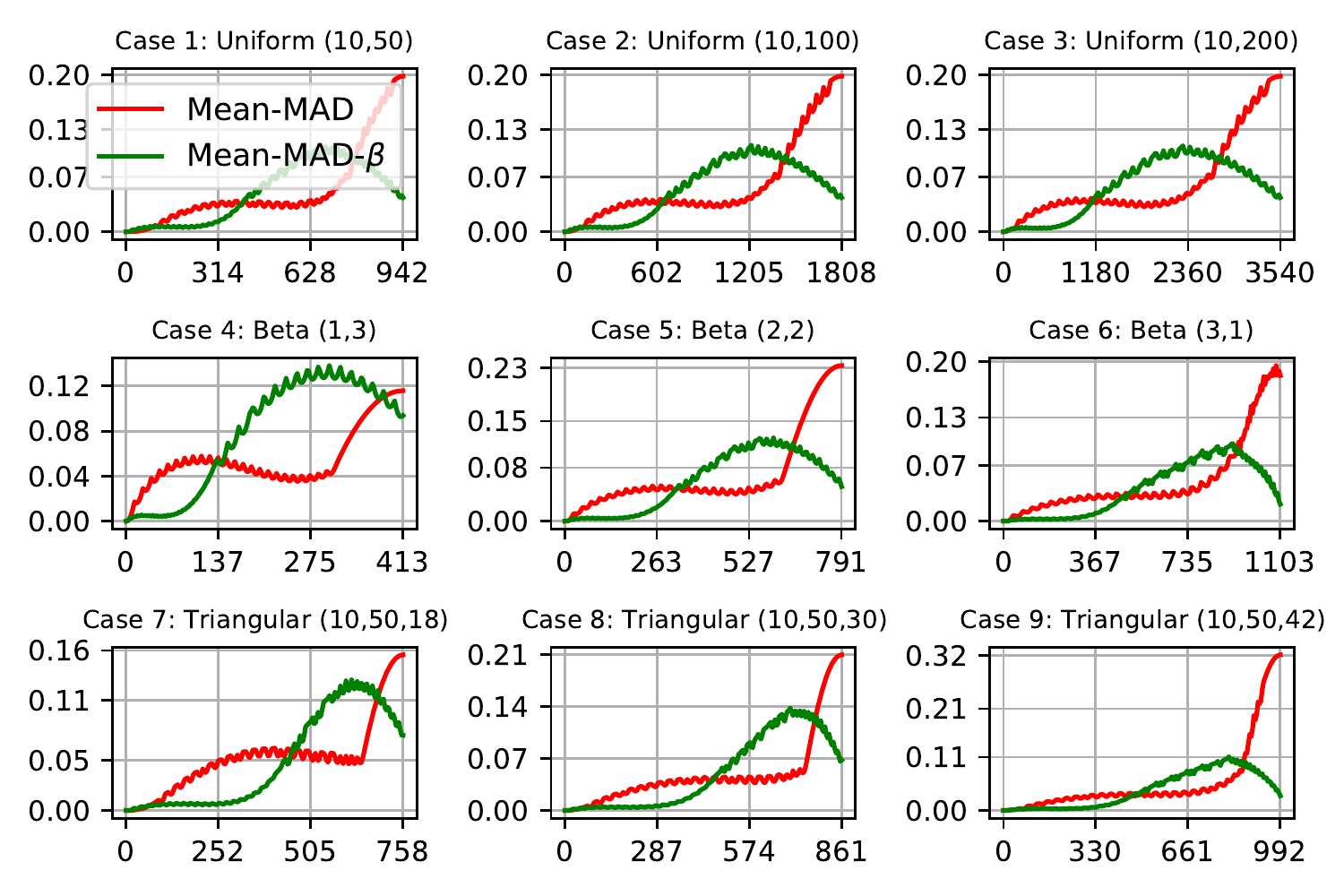}
    \caption{The results for the average margin setting. The x-axis corresponds to $B$ and the y-axis to the EVAI.}
    \label{performanceMADaverage}
\end{figure}

\subsection{Mean-variance comparison}\label{ec:meanvarcomp}
We start the performance analysis for the low margin scenario. The x-axis refers to the budget level $B$, and the y-axis refers to the EVAI. In each plot, the blue line corresponds to the EVAI for the mean-MAD model and the orange line to the mean-variance EVAI. Figure~\ref{performanceN25low} contains the performance plots for each of the nine cases we are considering.

\begin{figure}[h]
    \centering
    \includegraphics[width=.85\linewidth]{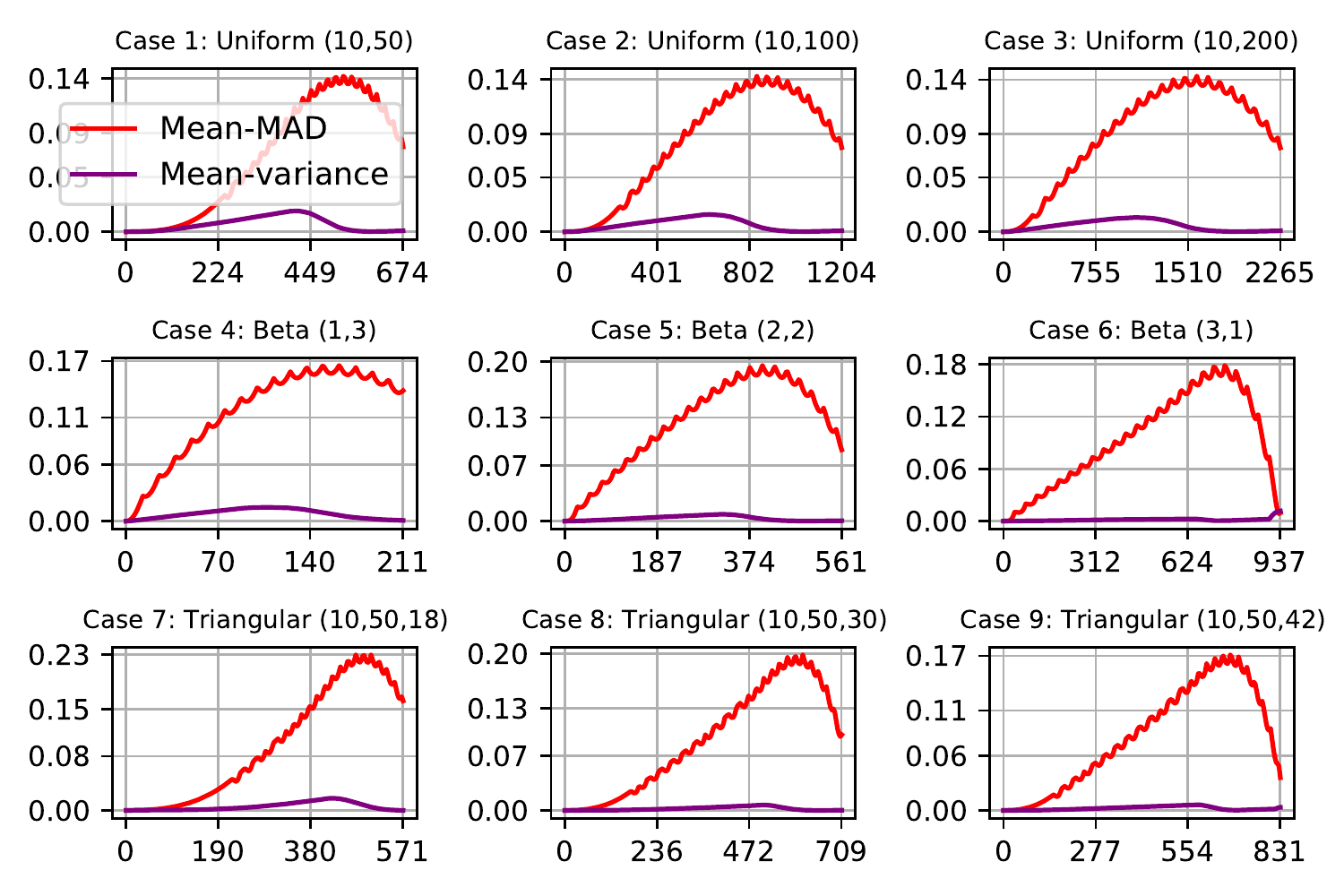}
    \caption{The results for the low margin scenario. The x-axis corresponds to the budget level and the y-axis to the EVAI.}
    \label{performanceN25low}
\end{figure}



In Figure~\ref{performanceN25low} we compare the mean-MAD policy with the mean-variance ordering policy in terms of EVAI for the scenario with low margins and a total of nine ground-truth demand distributions. 
While both policies generally give low EVAIs, the EVAI  
  of the mean-variance policy is typically lower.
We stress that this does not mean that the mean-variance policy is better. 
Indeed, a fair numerical comparison is impossible, as the respective ambiguity sets can contain vastly different distributions. While a finite variance excludes distributions with infinite-second moment, MAD does not. In general, the worst-case scenarios or extremal distributions are `more extreme' for MAD than for variance. This also offers a possible explanation for the slightly higher EVAI.

\section{Extensions}\label{sec:4}
We now present a distribution-free analysis for three extensions of the multi-item newsvendor model. 
Section~\ref{sec:moreconstraints} deals with multiple constraints, Section~\ref{sec:uncertainsupply} considers uncertain supply, and Section~\ref{sec:cvar} discusses the risk-averse newsvendor where the conditional value at risk (CVaR) is chosen as objective function.

\subsection{Multiple constraints}\label{sec:moreconstraints}
\citet{lau1996newsstand} consider the newsvendor problem with multiple constraints, and propose a numerical solution procedure that computes the Lagrange multipliers as roots of a system of nonlinear equations. \citet{perakis2020leveraging} also consider multiple capacity constraints in a retail environment, and distinguish between warehouse capacity and inventory availability constraints. By exploiting Lagrangian duality the problem is decomposed into two subproblems, which are solved iteratively by binary search. 

We now argue that the distribution-free analysis developed in the present paper also carries over to the setting with multiple 
constraints, and takes the form
\begin{equation}\label{eq:newsmaxform}
\begin{aligned}
\min _{\vecc{q}} & \sum_{i=1}^{n} c_{i}\left(d_{i} (q_{i}-\mu_i)+\left(m_{i}+d_{i}\right)\left(p^{(i)}_{1}\left(a_{i}-q_{i}\right)^{+}+p^{(i)}_{2}\left(\mu_{i}-q_{i}\right)^{+}+p^{(i)}_{3}\left(b_{i}-q_{i}\right)^{+}\right)\right)  \\
\text {s.t. } & \sum_{i=1}^{n} c_{i,j} q_{i} \leq B_j \quad  j=1,\ldots,m \\
& q_{i} \geq 0 \quad  i=1, \ldots, n.
\end{aligned}
\end{equation}
By introducing dummy variables $\tau^{(i)}_{k}$, we reformulate problem \eqref{eq:newsmaxform} as
\begin{equation}\label{eq:newslpform}
\begin{aligned}
\min _{\vecc{q},\boldsymbol{\tau}} &  \sum_{i=1}^{n} c_{i}\left(d_{i} (q_{i}-\mu_i)+\left(m_{i}+d_{i}\right) \left(p^{(i)}_{1}\tau^{(i)}_{1}+p^{(i)}_{2}\tau^{(i)}_{2}+p^{(i)}_{3}\tau^{(i)}_{3}\right)\right)   \\
\text {s.t. } & \sum_{i=1}^{n} c_{i,j} q_{i} \leq B_j, &  j=1,\ldots,m, \\
& \tau^{(i)}_{k} \geq \xi^{(i)}_{k}-q_{i}, & k=1,2,3; \; i=1, \ldots, n, \\
& \tau^{(i)}_{k} \geq 0, & k=1,2,3; \; i=1, \ldots, n, \\
& q_{i} \geq 0,  & i=1, \ldots, n,
\end{aligned}
\end{equation}
which remains a tractable LP,
solvable for large-scale problems with interior-point methods.
Moreover, by solving the dual problem of \eqref{eq:newslpform}, shadow prices of the $m$ budget constraints can be computed that quantify marginal expected net benefit of allocating an additional unit of budget to $B_j,\ j=1,\ldots m$. 

\subsection{Supply and demand uncertainty}\label{sec:uncertainsupply}
The newsvendor might take different decisions when the delivery of an order for $q$ units is not necessarily complete (uncertain supply). 
\citet{kaki2015newsvendor} consider  uncertain supply and uncertain demand, when supply and demand are independent or follow a particular copula-based dependency structure. In the mean-variance setting and under the independence assumption, \citet{gallego1993distribution} solve the distribution-free newsvendor problem with random yield, but assume the yield is a binomial random variable that depends on the order size $q$. That is, when an order for $q$ units is made, each individual unit is received with some fixed probability, or is not delivered at all.

As opposed to \citet{gallego1993distribution}, we do introduce an ambiguity set for the random supply. Consider the setting with multiplicative yield $Z_i$, where the random supply is given by $Z_i\cdot q_i$. Assume $Z_i$ 
has mean $\tilde{\mu_i}$, MAD $\tilde{\delta_i}$ and support $[\tilde{a}_i,\tilde{b}_i]$, where $0\leq \tilde{a}_i\leq \tilde{b}_i\leq1$. The distribution of $Z_i$ then resides in $\cP_{(\tilde{\mu}_i,\tilde{\delta}_i)}$. 
The extremal three-point distribution for $Z_i$ has probabilities
$$
\tilde{p}^{(i)}_{1}=\frac{\tilde{\delta}_{i}}{2\left(\tilde{\mu}_{i}-\tilde{a}_{i}\right)}, \quad \tilde{p}^{(i)}_{2}=1-\frac{\tilde{\delta}_{i}}{2\left(\tilde{\mu}_{i}-\tilde{a}_{i}\right)}-\frac{\tilde{\delta}_{i}}{2(\tilde{b}_{i}-\tilde{\mu}_{i})}, \quad \tilde{p}^{(i)}_{3}=\frac{\tilde{\delta}_{i}}{2(\tilde{b}_{i}-\tilde{\mu}_{i})},
$$
and is supported on $\zeta^{(i)}_{1}=\tilde{a}_i, \, \zeta^{(i)}_{2}=\tilde{\mu}_i, \,  \zeta^{(i)}_{3}=\tilde{b}_i,$ respectively.
The multi-item newsvendor with supply ambiguity is equivalent to
\begin{equation}\label{eq:uncertainsupply}
\begin{aligned}
\min _{\vecc{q}} & \sum_{i=1}^{n} \max_{\P\in\cP_i} \E_\P \left[c_{i}\left(d_{i} (Z_i\cdot q_{i} - D_i)+\left(m_{i}+d_{i}\right)(D_i - Z_i\cdot q_i)^+\right)\right]  \\
\text {s.t. } & \sum_{i=1}^{n} c_{i} q_{i} \leq B_j, \quad  j=1,\ldots,m, \\
& q_{i} \geq 0, \quad  i=1, \ldots, n,
\end{aligned}
\end{equation}
with $\cP_i:=\cP_{(\mu_i,\delta_i)}\times\cP_{(\tilde{\mu}_i,\tilde{\delta}_i)}$. Since the newsvendor problem is jointly convex in the pairwise independent random variables $D_i$ and $Z_i$, the distributions that maximize the objective function of \eqref{eq:uncertainsupply} are the extremal three-point distributions. Applying these worst-case distributions to \eqref{eq:uncertainsupply} results in
\begin{equation}\label{eq:lpuncertainsupply}
\begin{aligned}
\min _{\vecc{q}} &  \sum_{i=1}^{n} c_{i}\Big(d_{i} (\tilde{\mu_i}q_{i} - \mu_i)+\left(m_{i}+d_{i}\right)\! \sum_{\boldsymbol{\kappa}\in\{1,2,3\}^2}p^{(i)}_{\kappa_1}  \tilde{p}^{(i)}_{\kappa_2}  \tau^{(i)}_{\boldsymbol{\kappa}}\Big)   \\
\text {s.t. } & \sum_{i=1}^{n} c_i q_{i} \leq B, &   \\
& \tau^{(i)}_{\boldsymbol{\kappa}} \geq \xi^{(i)}_{\kappa_1}-\zeta^{(i)}_{\kappa_2}q_{i}, & \bm{\kappa}\in\{1,2,3\}^2; \; i=1, \ldots, n, \\
& \tau^{(i)}_{\bm{\kappa}} \geq 0, & \bm{\kappa}\in\{1,2,3\}^2; \; i=1, \ldots, n, \\
& q_{i} \geq 0,  & i=1, \ldots, n.
\end{aligned}
\end{equation}

To demonstrate the distribution-free newsvendor with uncertain supply, consider the one-dimensional case with random demand $D$ with a uniform distribution on $[20,80]$ and multiplicative yield $Z$ uniformly distributed on $[0.65,0.95]$. Figure~\ref{fig:supplyexample} depicts the tight lower and upper bounds that follow from optimizing over the ambiguity sets that contain the distributions of $D$ and $Z$. As the extremal distributions are discrete, the objective function of \eqref{eq:lpuncertainsupply} admits a piecewise linear representation. 

\begin{figure}[h!]
    \centering
    \includegraphics[width=0.6\textwidth]{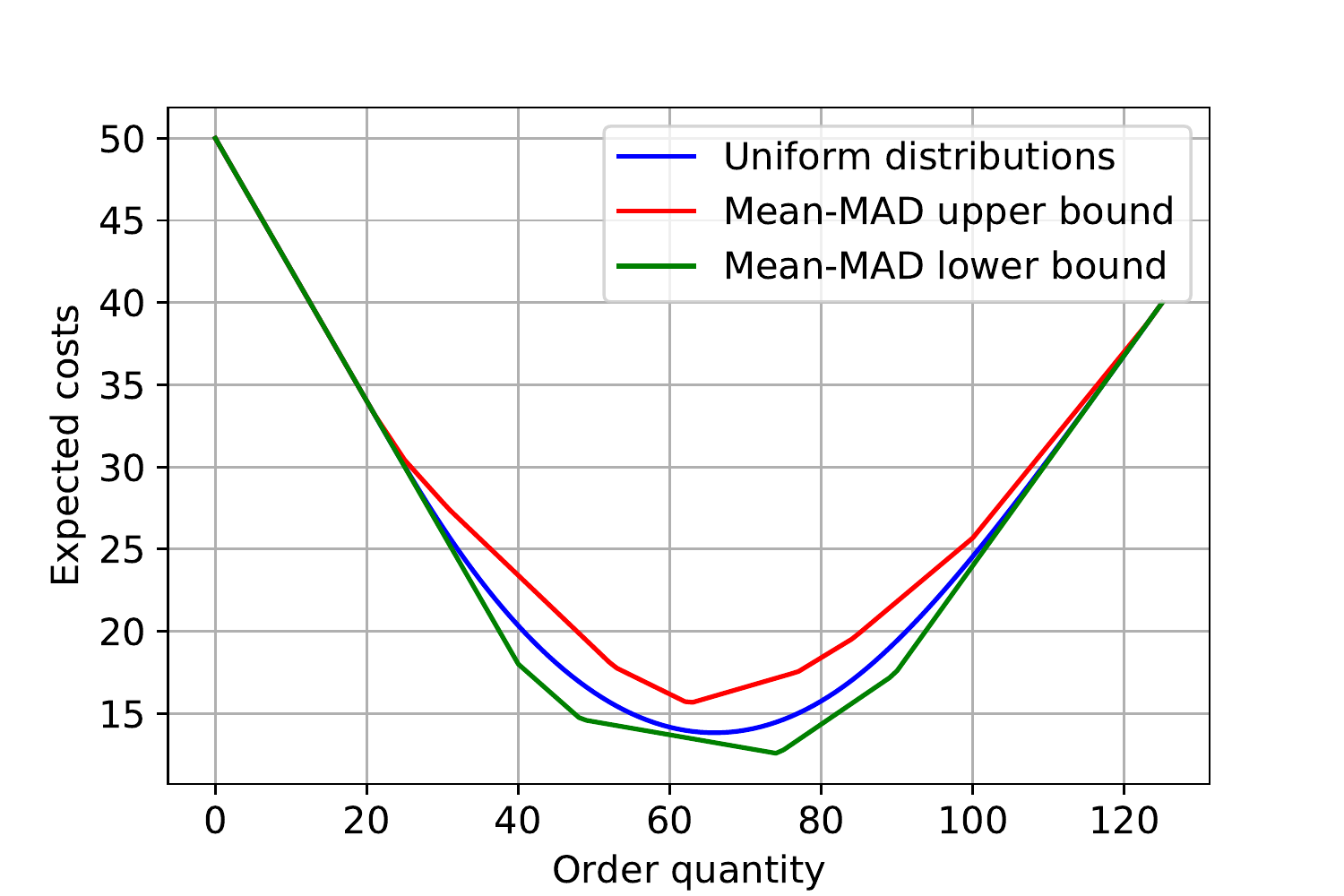}
    \caption{Tight bounds for the multi-item newsvendor with uncertain supply yield, where $m=1$ and $d=0.8$. The upper piecewise linear function is obtained by evaluating $\E[D-Z\cdot q]$, with $D$ following the extremal distribution that lies in $\cP_{(50,15,20,80)}$ and $Z$ the worst-case three-point distribution in $\cP_{(0.8,0.075,0.65,0.95)}$. The lower bound follows from the best-case two-point distributions. The middle curve depicts the ‘true’ costs, where $D$ has a uniform distribution on $[20,80]$, and $Z$ is uniformly distributed on $[0.65,0.95]$.}
    \label{fig:supplyexample}
\end{figure}

Because problem \eqref{eq:uncertainsupply} can be written in terms of a piecewise linear function, the optimal solution follows from a knapsack algorithm similar to Theorem~\ref{thm:AllocationAlgorithm}. Further, one can gain additional insights by explicitly deriving the optimal order quantities for the robust single-item model, as in Theorem~\ref{thm:singleMadStrategy}. The problem is similar for additive yield, also resulting in a three-point distribution for the worst case. 
Other directions for future research include solving \eqref{eq:uncertainsupply} with multiple unreliable and non-identical suppliers \citep{dada2007newsvendor} and the newsvendor problem with fixed ordering costs and supplier capacity restrictions \citep{merzifonluoglu2014newsvendor}.

\subsection{Risk aversion}\label{sec:cvar}
We next consider a risk-averse decision maker, as in \citet{chen2010cvar}, who makes decisions based on $\operatorname{CVaR}$. The decision maker no longer optimizes the expected costs, but instead minimizes the average value of the costs exceeding the $\gamma$th-quantile of the newsvendor's cost distribution. For the cost function $G(\vecc{q},\vec{D})$, $\operatorname{CVaR}$ can be calculated by solving a convex minimization problem  \citep{rockafellar2000optimization}:
$$
\min_{\theta\in\mathbb{R}}\left\{\theta + \frac{1}{1-\gamma}\E(G(\vecc{q},\vec{D})-\theta)^+\right\}.
$$
Calculating $\operatorname{CVaR}$ requires full knowledge of the demand distribution. However, in practice, committing to a particular distribution might be problematic for the decision maker if there is not enough data available. Hence, we consider the partial information setting as in \citet{zhu2009worst,Delage2010}, and seek to solve
\begin{equation}\label{eq:CVaR}
\min_{\vecc{q}:\sum_i c_iq_i\leq B,q_i\geq0}\, \max_{\P\in\cP_{(\mu,\delta)}}\min_{\theta\in\mathbb{R}}\left\{\theta + \frac{1}{1-\gamma}\E_{\P}(G(\vecc{q},\vec{D})-\theta)^+\right\}.
\end{equation}
Let us first consider the single-item model. Because the objective function of \eqref{eq:CVaR} is finite, $\cP_{(\mu,\delta)}$ is weakly compact as $\text{supp}(D)$ is compact, and the objective function of \eqref{eq:CVaR} is linear in $\P$ and convex in $\theta$, we are allowed to interchange the maximization and minimization operators by virtue of the minimax theorem \citep{shapiro2002minimax}. 
Since $\left(G(q,D)-\theta\right)^+$ is a convex function of the uncertain demand, the three-point distribution \eqref{eq:BenTal_probabilities_multivariable} also maximizes $\E_{\P}(G(q,D)-\theta)^+$. 
When $\beta=\P(D\geq\mu)$ is known, the two-point distribution in Lemma~\ref{lemma:meanmadbeta} attains the matching lower bound. For the multivariate problem, 
notice that $(G(\vecc{q},\vec{D})-\theta)^+$ is again a convex function of the uncertain demand, where $\vec{D}\sim\P\in\cP_{(\mu,\delta)}$. By Proposition~\ref{postekprop} and the reasoning above, the risk-averse newsvendor admits the following LP representation:
\begin{equation}\label{eq:lpriskaverse}
\begin{aligned}
\min _{\vecc{q},\boldsymbol{\tau},\boldsymbol{\eta},\theta}\  &  \theta + \frac{1}{1-\gamma} \sum_{\boldsymbol{\kappa}\in\{1,2,3\}^n}\prod\limits_{i=1}^{n}\, p_{\kappa_i}^{(i)}  \eta_{\boldsymbol{\kappa}}   \\
\text {s.t. } & \sum_{i=1}^{n} c_i q_{i} \leq B, &   \\
& \eta_{\boldsymbol{\kappa}} \geq \Big(\sum_{i=1}^{n} c_{i}\left(d_{i} (q_{i}-\xi_{\kappa_i}^{(i)})+\left(m_{i}+d_{i}\right)\tau^{(i)}_{\boldsymbol{\kappa}}\right)\Big)-\theta, & \boldsymbol{\kappa}\in\{1,2,3\}^n, \\
& \eta_{\boldsymbol{\kappa}} \geq0, & \boldsymbol{\kappa}\in\{1,2,3\}^n, \\
& \tau^{(i)}_{\boldsymbol{\kappa}} \geq \xi^{(i)}_{\kappa_i}-q_{i}, & \boldsymbol{\kappa}\in\{1,2,3\}^n; \; i=1, \ldots, n, \\
& \tau^{(i)}_{\boldsymbol{\kappa}} \geq 0, & \boldsymbol{\kappa}\in\{1,2,3\}^n; \; i=1, \ldots, n, \\
& q_{i} \geq 0,  & i=1, \ldots, n.
\end{aligned}
\end{equation}

We show in Figure~\ref{fig:exampleCVaR} the bounds for the single-item model with demand having support $[10,50]$, $\mu=30$, $\delta=20/3$ and $\beta=1/2$. 
Solving $\eqref{eq:lpriskaverse}$ for $\gamma=0.75,0.95$ and different order sizes yields the upper bounds. We solve an analogous problem, but with the expectation taken over the extremal two-point distribution, stated in Lemma~\ref{lemma:meanmadbeta}, to obtain the tight lower bounds. As a point of reference, we also plot the exact values of the CVaR and expected costs when $D$ follows a symmetric triangular distribution on $[10,50]$.

\begin{figure}[h]
\centering
\begin{subfigure}{.5\textwidth}
  \centering
  \includegraphics[width=\linewidth]{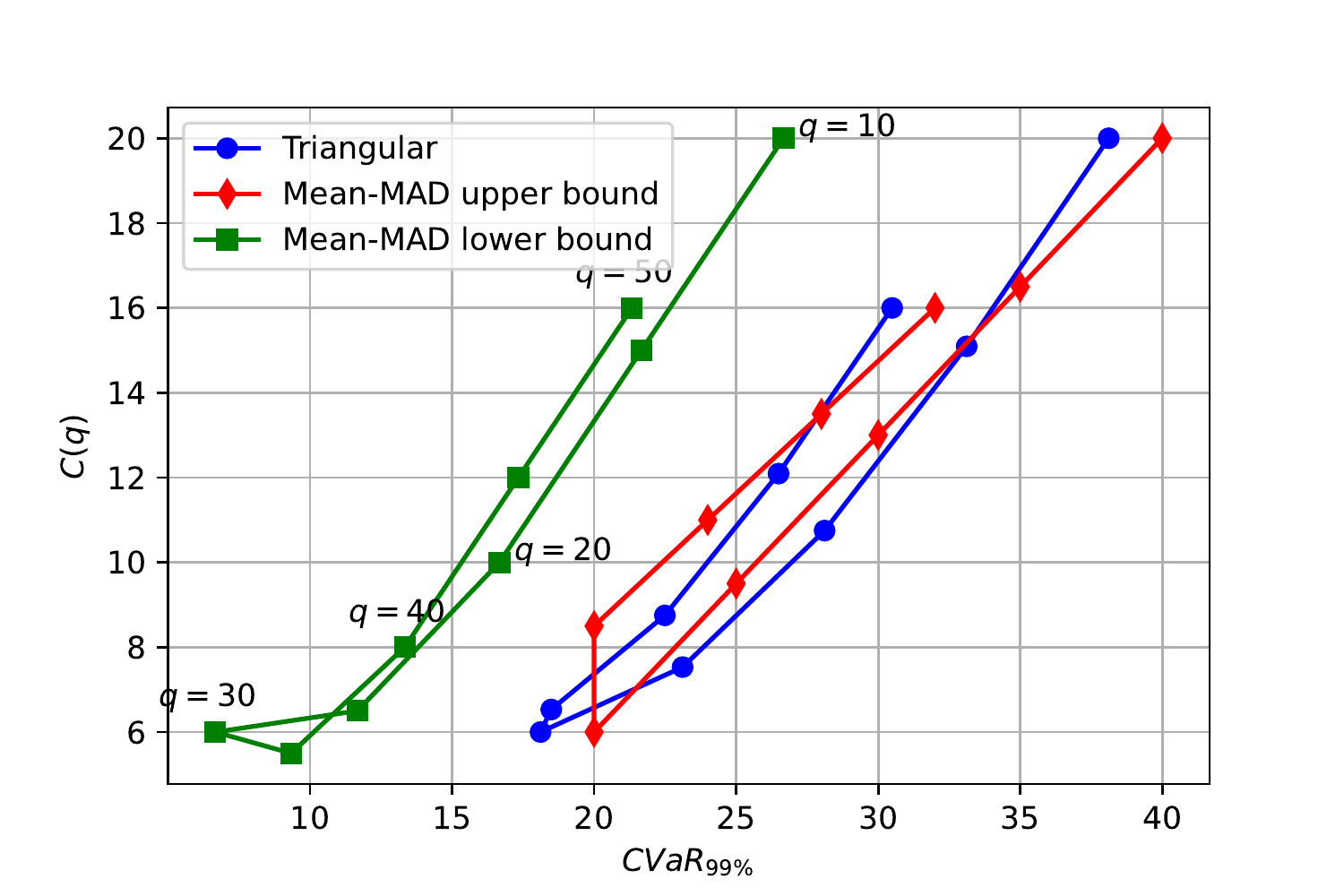}
    \caption{Expected costs and CVaR}
  \label{fig:exCVaRcosts}
\end{subfigure}%
\begin{subfigure}{.5\textwidth}
  \centering
  \includegraphics[width=\linewidth]{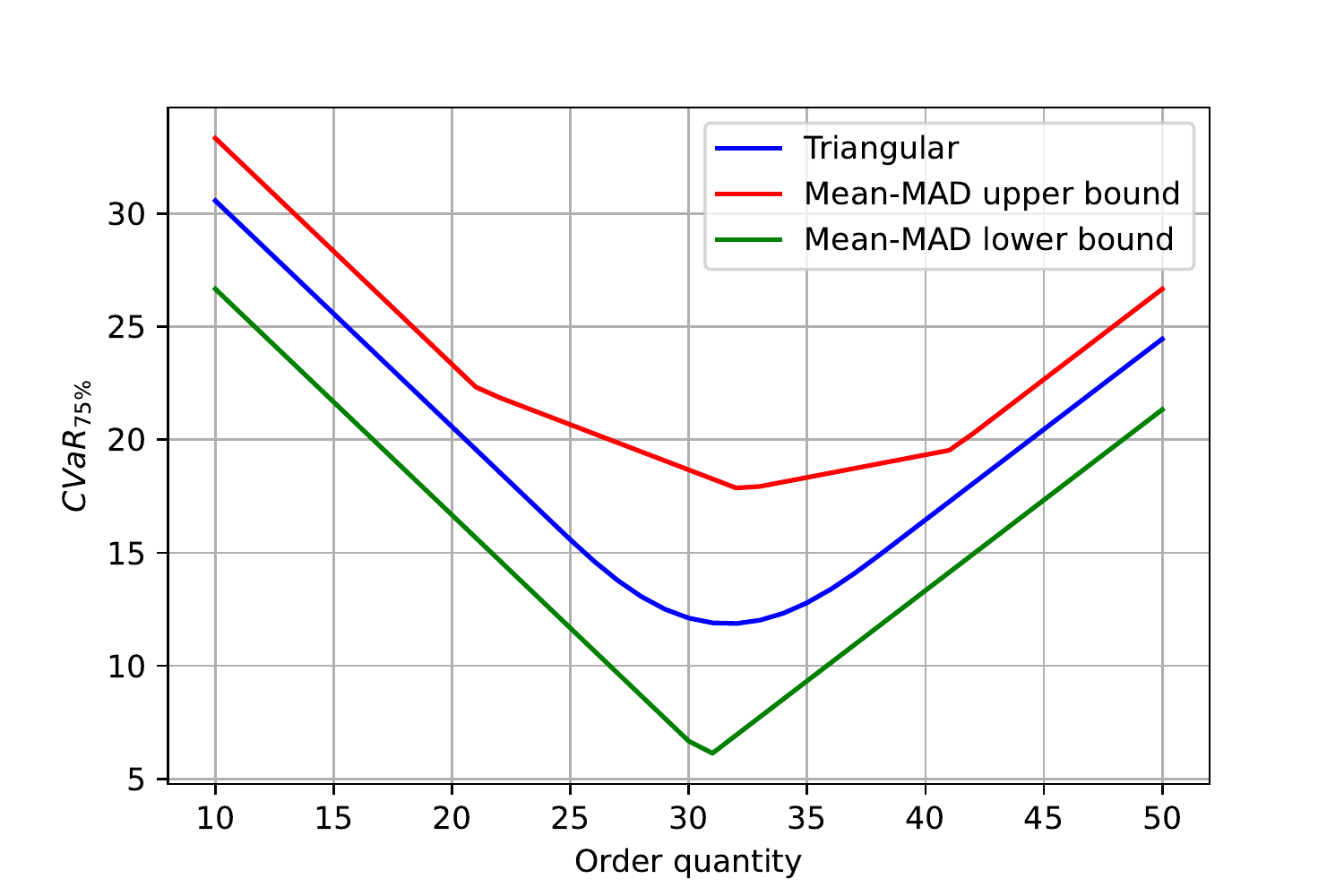}
  \caption{Mean-MAD  bounds for CVaR}
  \label{fig:exCVaRbounds}
\end{subfigure}
\caption{An illustration of the tight mean-MAD bounds for the risk-averse newsvendor with CVAR as objective criterion, where $m=1$, $d=0.8$ and $\gamma=0.75,0.99$. The middle curve corresponds to the CVaR when $D$ follows a symmetric triangular distribution on $[10,50]$. The upper and lower bounds follow from optimizing over the ambiguity sets that contain this distribution.}
\label{fig:exampleCVaR}
\end{figure}

Solving \eqref{eq:lpriskaverse} can be challenging since the objective function $(G(\vecc{q},\vec{D})-\theta)^+$ is no longer separable, thus resulting in an exponential number of variables and constraints. To alleviate this computational difficulty, one might resort to sampling-based procedures such as sample average approximation \citep{shapiro2009lectures}. 

We also mention ambiguous chance constraints that can be conservatively approximated by CVaR \citep{nemirovski2007convex}. In the risk-averse newsvendor setting, the decision maker introduces an ambiguous chance constraint that restricts the probability of the costs exceeding a certain threshold $t$ to be less than $1-\gamma$, considering all distributions in the ambiguity set. For the multi-item setting, this means ensuring
$$
\P(G(\vecc{q},\vec{D})>t)\leq1-\gamma, \quad \forall\P\in\cP_{(\mu,\delta)},
$$
which is implied by
$$
\max_{\P\in\cP_{(\mu,\delta)}}\operatorname{CVaR}_{\gamma}[G(\vecc{q},\vec{D})]\leq t.
$$
In addition, the newsvendor might require a minimal probability that all customer orders will be completely covered by the inventory on hand, i.e., the type-1 service level \citep{silver1998inventory}. When several of these probabilistic constraints are interrelated, the decision maker should conservatively approximate joint chance constraints. For this one can again use CVaR; see \citet{chen2010cvar,zymler2013distributionally,roos2020reducing}.
Adding ambiguous chance constraints to the models developed in this paper is a worthwhile topic for further research.






\end{document}